\documentclass[12pt,reqno]{amsart}  

\usepackage{amsmath,amsthm,amssymb,amscd,mathdots,mathtools,enumerate,float,shuffle,stmaryrd}

\usepackage{hyperref}
\usepackage{tikz,tikz-cd}
\allowdisplaybreaks
\usetikzlibrary{decorations.pathmorphing,shapes,calc}

\usepackage{indentfirst}

\mathtoolsset{showonlyrefs}
\usepackage{fancyhdr}
\newtheorem{theorem}{Theorem}[section]
\newtheorem*{theorem*}{Theorem}

\newtheorem{lemma}[theorem]{Lemma}

\numberwithin{equation}{section}

\theoremstyle{definition}

\newcommand{\Z}{\mathbb{Z}}

\DeclareMathOperator{\arcsinh}{arcsinh}

\definecolor{mycolor}{RGB}{194, 8, 88}
\newcommand{\todo}[1]{\message{LaTeX Warning: You did not finish your work :-( on input line \the\inputlineno} {\color{mycolor} {\big[\,}{\bf Todo:} #1\,\big]}}

\makeatletter
\newcommand{\cop}{\DOTSB\cop@\slimits@}
\newcommand{\cop@}{\mathop{\bigstar}}
\newcommand{\cop@@}[2]{%
  \vphantom{\sum}%
  \ifx#1\displaystyle\big#2\else#2\fi
}
\makeatother

\title{The mean square of the product of the Riemann zeta-function and a Dirichlet polynomial in the critical strip}
\author{Yu Jinbo}

\address{Graduate School of Mathematics, Nagoya University, Nagoya, Japan.}
\email{jinbo.yu.e6@math.nagoya-u.ac.jp}

\begin{document}
\date{\today}
\maketitle
%\tableofcontents
\begin{abstract} 
We refine a previous work of K. Matsumoto and H. Ishikawa \cite{IM}, obtaining an asymptotic formula for the mean square of the product of the Riemann zeta-function and a Dirichlet polynomial in the critical strip ($1/4<\sigma<1/2$), by obtaining an explicit formula of Atkinson type for its error term. This work is closely related to the generalized Dirichlet divisor problem. The form of these formulas is akin to, yet more complex than, Voronoi's formulas in the divisor problem.
\end{abstract}

\section{Introduction}
Let $s=\sigma+it$ be a complex variable, $\zeta(s)$ be the Riemann zeta-function, and $$A(s)=\sum_{m=1}^M a(m)m^{-s}$$be a Dirichlet polynomial, where $1\le M\in \Z$, $a(m)\in\mathbb{C}$, with $a(m)=O(m^{\epsilon})$ for $\varepsilon>0$.

In number theory, a significant area of study is the asymptotic behavior of certain mean values. In particular, for $T\geq 2$ the asymptotics for the mean value
\begin{align}
     I\left(T,A;\frac{1}{2}\right)=\int_0^T|\zeta(1/2+it)A(1/2+it)|^2dt\ %\quad (T \geq 2)
    \label{1.1}
\end{align}
are of great interest. These have been used to understand the distribution of the values of $\zeta(s)$ and the location of its zeros \cite{S}.
It is crucial to allow $M$ to be as large as possible. If $M=T^{1-\varepsilon}$ could be taken in \eqref{1.1} to obtain a sufficiently sharp asymptotic formula, it would imply the Lindel\"of Hypothesis. 

Balasubramanian, Conrey and Heath-Brown \cite{BCH} first obtained an asymptotic formula for $I\left(T,A;\frac{1}{2}\right)$. Under the condition $\log M\ll \log T$ they show that
\begin{align*}
     I\left(T,A;\frac{1}{2}\right)=\sum_{k\le M}\sum_{l\le M}\frac{a(k)\overline{a(l)}}{[k,l]}\left(\log\frac{(k,l)^2T}{2\pi kl}+2\gamma-1\right)T +E\left(T,A;\frac{1}{2}\right),
\end{align*}
where $\overline{a(l)}$ is the complex conjugate of $a(l)$, $(k,l)$ means the greatest common divisor of $k$ and $l$, $[k,l]$ means the least common multiple and $\gamma$ is Euler's constant. The error estimate is $$E\left(T,A;\frac{1}{2}\right)\ll M^2T^{\varepsilon}+T\log^{-c}T$$ for any $c>0$. Their motivation was to understand the location of the zeros of the Riemann zeta-function. Specifically, they deduced that at least 38$\%$ of the complex zeros of $\zeta(s)$ are on the critical line $\Re(s)=\frac{1}{2}$ by selecting the special form of $a(m)$. And if their conjecture about the error estimate holds (a better result), then this number can be raised to 55$\%$.

Ishikawa and Matsumoto \cite{IM} prove an analogue of Atkinson's formula for $E\left(T,A;\frac{1}{2}\right)$. Atkinson \cite{A} was the first to give an explicit formula for $E(T)$ defined by 
$$E(T)= \int_0^T\left|\zeta\left(\frac{1}{2}+it\right)\right|dt-T\log T-(2\gamma-1-\log 2\pi)T.$$
It is very useful in the mean square theory of $\zeta(s)$. Naturally, it also applies to other zeta-functions, L-functions, and other functions that are related to zeta-functions. Matsumoto \cite{M} applied Atkinson's idea to study the analog of $E(T)$ in the critical strip($1/2<\sigma<3/4$), reinforced the connection between the mean square of Riemann zeta-function and the Voronoi-type formula for the divisor function. Ishikawa \cite{I3} introduced a new idea to prove an explicit formula of Atkinson type for Dirichlet L-function $L(s;\chi)$ associated with a complex character $\chi$. The fundamental idea, inspired by Hafner and Ivi\'{c} \cite{I1}, is to consider the integral over the interval $[T,2T]$ instead of $[-T,T]$. This is to avoid the integral over the interval $[-T, T]$, which does not necessarily equal half of the integral over $[0, T]$. For example, in the $\zeta(s)$ case, $$\int_0^T\left|\zeta\left(\frac{1}{2}+it\right)\right|dt=\frac{1}{2}\int_{-T}^T\left|\zeta\left(\frac{1}{2}+it\right)\right|dt,$$
but in our case, $I(T,A;1/2)$ is not necessarily equal to $$\frac{1}{2}\int_{-T}^T\left|\zeta\left(\frac{1}{2}+it\right)A\left(\frac{1}{2}+it\right)\right|^2dt.$$

In the present paper, we try to understand the asymptotic formula of 
$$I(T,A;\sigma) = \int_0^T|\zeta(\sigma+it)A(\sigma+it)|^2dt\ \ for\ 2\le T$$
on the critical strip $(1/4<\sigma<1/2)$ which has not been discussed yet. Our argument is inspired by Ishikawa and Matsumoto \cite{IM}, but further, we apply the idea from Ishikawa \cite{I3} to $E(T,A;\sigma)$ to obtain an analogue of Atkinson's formula. And reinforced the connection between the mean square of the Riemann zeta-function and the Voronoi-type formula for a generalization of the divisor function like Matsumoto \cite{M}.

First, we introduce some notation, $k$ and $l$ are positive integers and $T,u>0$: 
$$\kappa=\frac{k}{(k,l)},\ \ \ \lambda=\frac{l}{(k,l)},\ \ \ e(x)=\exp(2\pi i x)$$$$\arcsinh x=\log(x+\sqrt{x^2+1}),\ \ \xi(T,u)=\frac{T}{2\pi}+\frac{u}{2}-\sqrt{\frac{u^2}{4}+\frac{uT}{2\pi}}$$
$$f(T,u)=2T\arcsinh\sqrt{\frac{\pi u}{2T}}+\sqrt{2\pi u T-\pi^2u^2}-\frac{\pi}{4},\ \ g(T,u)=T\log\frac{T}{2\pi u}-T+2\pi u+\frac{\pi}{4}$$
Let us define
$$\mathcal{M}(T,A)=\sum_{k\le M}\sum_{l\le M}\frac{a(k)\overline{a(l)}}{[k,l]^{2\sigma}}\left[\zeta(2\sigma)T+\Gamma(2\sigma-1)\zeta(2\sigma-1)[k,l]^{2\sigma-1}\frac{\cos((\sigma-1/2)\pi)}{1-\sigma}T^{2-2\sigma}\right],$$

$$\Sigma_1(T,Y)=\sum_{k,l \le M}\sum_{n\le \kappa\lambda Y}\Im\left\{\frac{a(k)\bar{a(l)}}{[k,l]^{2\sigma}}(\kappa\lambda)^{\sigma}\exp(-2i\pi\sigma)\sigma_{2\sigma-1}(n)n^{-\sigma}e(\bar{\kappa}n/\lambda)T^{1/2-\sigma}\right.$$
$$\left.\times \left(\arcsinh\sqrt{\frac{\pi n}{2T\kappa\lambda}}\right)^{-1}(1+2T\kappa\lambda/\pi n)^{-1/4}\exp(i(f(T,n/\kappa\lambda)-\pi n/\kappa\lambda+\pi/2))\right\}.$$
where $\sigma_{a}(n)$ is the sum of $a$-powers of positive divisors of $n$, $\bar{\kappa}$  is defined by $\kappa\bar{\kappa}\equiv1(\mod\lambda)$, and
$$\Sigma_2(T,Y)=-\sum_{k,l \le M}\sum_{n\le (\lambda/\kappa) Y}\frac{T^{1/2-\sigma}}{\pi^{1/2+\sigma}2^{\sigma-1/2}}\Re\left\{\frac{a(k)\bar{a(l)}}{[k,l]^{2\sigma}}\kappa^{\sigma}\lambda^{\sigma}\frac{\sigma_{2\sigma-1}(n)e(-\kappa n/\lambda)}{n^{\sigma}}\frac{4\pi \exp\left\{ig(T,\kappa n/\lambda)\right\}}{\log(\lambda T/2\pi \kappa n)}\right\}.$$

\begin{theorem}\label{thm1}
    Let $T,Y > 0$ with $C_1T<Y<C_2T$ and $T\ge C^{\star}=\max\{e,\ C_1^{-1}\}$, where $C_1$ and $C_2$ are fixed constants with $0<C_1<C_2$. Then for $1/4<\sigma<1/2$ we have
    $$\int_T^{2T} |\zeta(\sigma+it)A(\sigma+it)|^2dt=\mathcal{M}(2T,A)+\Sigma_1(2T,2Y)+\Sigma_2(2T,\xi(2T,2Y))$$
$$-\mathcal{M}(T,A)-\Sigma_1(T,Y)-\Sigma_2(T,\xi(T,Y))+R(T,2T,A),$$
where $R(T,2T,A)$ is the error term satisfying
$$R(T,2T,A)\ll_{M} T^{1-2\sigma}\log(T).$$
\end{theorem}
\begin{theorem}\label{thm2}
    Under the same assumptions as in Theorem\ref{thm1}, we have
    \begin{align*}
        E(T,A;\sigma)=I(T,A;\sigma)-\mathcal{M}(T,A)=\Sigma_1(T,Y)+\Sigma_2(T,\xi(T,Y))+R(T,A)
    \end{align*}
    with 
    \begin{align*}
        R(T,A)\ll _M\log^{(2-2\sigma)\alpha}(T)+\log^{(3/2-2\sigma+\varepsilon)\alpha}(T)+\log^{\alpha}(T)\log T+ T^{1-2\sigma}\log^2 T
    \end{align*}
    for any $\alpha>0$.
\end{theorem}

The results give an asymptotic formula for the mean square of the product of the Riemann zeta-function and a Dirichlet polynomial in the left-critical strip and an Atkinson's type explicit formula for $ E(T,A;\sigma)$. 
It is important to note that the discussion in this paper does not apply to the right-critical strip $1/2<\sigma<3/4$. In this paper we use Euler–Maclaurin formula for the product of the Riemann zeta-function and Dirichlet polynomial. We believe that if we use Poisson summation formula as in \cite{A}, we can reach the conclusion in the right-critical strip $1/2<\sigma<3/4$ through a similar discussion. And in the form of the final result, they will show some kinds of similarity. 
%In this article, we use Euler–Maclaurin formula for the product of the Riemann zeta-function and Dirichlet polynomial. The next step in my research program is to discuss the results of the critical strip in the right half $(1/2<\sigma<3/4)$. We will accomplish this by carrying out the Poisson summation formula like Atkinson's original paper \cite{A} did.
%In the next part of the plan, we will discuss the results of the critical strip in the right half $(1/2<\sigma<3/4)$. We think that the ideas in this part will differ considerably from the main tool in this paper, where the treatment of the product function is essentially carried out by means of the Euler–Maclaurin formula, whereas in the right-half case we will accomplish this by carrying out the Poisson summation formula like Atkinson's original paper did. 

%%%%%%%%%%%%%%或许可以放在最后
%And look forward to lifting the $\log M\ll \log T$ restriction by introducing a smooth function $\phi\left(\frac{t}{T}\right)$.
%%%%%%%%%%%%%%%%%%
\section{A generalization of Dirichlet division problem}
Let $\sigma_a(n)$ be the divisor function, and denote 
$$D_{2\sigma-1}(x,\frac{\bar{\kappa}}{\lambda})=\sideset{}{'}\sum_{n\le x}\sigma_{2\sigma-1}(n)e\left(\frac{\bar{\kappa }n}{\lambda}\right),
$$
where the symbol $\sideset{}{'}\sum$ means that the last term is to be halved if $x$ is an integer. 
Kiuchi \cite{K} gave a results on $D_{-a}(n,x)$ for ($0<a<1/2$)
\begin{lemma}
For $0<a<1/2$, we have
\begin{equation}\begin{aligned}
\sideset{}{'}\sum_{n\le x}\sigma_{-a}(n)e\left(\frac{hn}{k}\right)=k^{-a-1}\zeta(1+a)x+k^{1+a}\frac{\zeta(1-a)}{1-a}x^{1-a}
+E_{-a}(0;h/k)+\Delta_{-a}(x;h/k),
\label{eq:D}
\end{aligned}\end{equation}
where, for $-1<a\le0$, $1\ll x$, $1<N\ll x$,$$E_a(s;h/k)=\sum^{\infty}_{n=1}\sigma_a(n)e(hn/k)n^{-s}, \ \ \Re{s}>1$$and with $N=k^{2/(3-2a)}x^{(1-2a)/(3-2a)}$, $k\le x$ we have
\begin{equation}
\Delta_a(x;h/k)\ll k^{(2-2a)/(3-2a)+\varepsilon}x^{1/(3-2a)+\varepsilon}.
\label{O:d}
\end{equation}

\label{lem-1}
\end{lemma}
Later, Roy \cite{R} derived an extended form of Voronoi summation formulas through some extensions from Oppenheim's article \cite{O}.
\begin{lemma}
For $-1/2<a<0$, we have
\begin{equation}\begin{aligned}
\Delta_a(x;h/k)=&
-2\pi^{-1}\cos(\pi a/2)x^{(1+a)/2}\sum_{n>0}\sigma_a(n)e\left(-\frac{hn}{k}\right)n^{-(1+a)/2}\\
&\times[K_{a+1}(4\pi\sqrt{nx}/k)+1/2\pi Y_{a+1}(4\pi\sqrt{nx}/k)]\\
&-\sin(\pi a/2)x^{(1+a)/2}\sum_{n>0}\sigma_a(n)e\left(-\frac{hn}{k}\right)n^{-(1+a)/2}J_{1+a}(4\pi\sqrt{nx}/k)\\
&+O(x^{-3/4+a/2}),
\label{eq:d}
\end{aligned}\end{equation}
\end{lemma}
where $K_{a+1}$, $Y_{a+1}$ and $J_{a+1}$ are the usual notations of Bessel functions.

\section{Preparation of proofs}
Assuming $u,v\in \mathbb{C}$, for $\Re{u}>1,\Re{v}>1$, we have:

$$\zeta(u)A(u)=\sum_{m=1}^{\infty}m^{-u}\sum_{k\le M}a(k)k^{-u}=\sum_{q=1}^{\infty}\left(\sum_{k\le M; k|q}a(k)\right) q^{-u}.$$
Therefore we have
$$ \zeta(u)\zeta(v)A(u)\overline{A(\bar{v})}=\sum^{\infty}_{q=1}\left(\sum_{l\le M,k|q}a(k)\right)q^{-u}\sum^{\infty}_{r=1}\left(\sum_{l\le M,k|r}a(k)\right)r^{-v}$$$$=\sum_{q=r}+\sum_{q<r}+\sum_{q>r}=B_0+B(u,v)+\overline{B(\bar{v},\bar{u})},$$
where $$B_0=\sum_{k\le M}\sum_{l\le M}a(k)\overline{a(l)}\sum_{[k,l]|r}r^{-u-v}=\zeta(u+v)\sum_{k\le M}\sum_{l\le M}\frac{a(k)\overline{a(l)}}{[k,l]^{u+v}}$$
and 
 $$B(u,v)=\sum_{m=1}^{\infty}\sum_{n=1}^{\infty}b(m)\overline{b(m+n)}m^{-u}(m+n)^{-v}$$
 with $$b(m)=\sum_{k\le M,k|m}a(k).$$
Letting $m=hk$, then
\begin{align*}
B(u,v)=&\sum_{m=1}^{\infty}\sum_{n=1}^{\infty}b(m)\overline{b(m+n)}m^{-u}(m+n)^{-v}\\
&=\sum_{k\le M}\sum_{l\le M}\frac{a(k)\overline{a(l)}}{k^u}\sum_{h,n\ge1;l|(hk+n)}\frac{1}{h^u(hk+n)^v}\\
&=\sum_{k\le M}\sum_{l\le M}\frac{a(k)\overline{a(l)}}{k^ul}\sum_{h=1}^{\infty}\frac{1}{h^u}\sum_{n=1}^{\infty}\frac{1}{(hk+n)^v}\sum_{f=1}^le^{2\pi i(hk+n)f/l}.
\end{align*}
Since the summation and the integration are absolute convergence in $y>0,\ \Re{v}>1,\ \Re{(u+v)}>2$, we have
\begin{equation}\begin{aligned}\Gamma(v)\sum_{n=1}^{\infty}e^{2\pi i fn/l}(hk+n)^{-v}&=\sum^{\infty}_{n=1}e^{2\pi ifn/l}\int^{\infty}_0y^{v-1}e^{-(hk+n)y}dy\\
&=\int^{\infty}_0\frac{y^{v-1}e^{-hky}}{e^{y-2\pi if/l}-1}dy,
\label{eq1}
\end{aligned}\end{equation}
and, for $\Re{u}>0,y>0$,
\begin{align*}\begin{aligned}k^{-u}\Gamma(u)\sum^{\infty}_{h=1}e^{-hky+2\pi ifhk/l}h^{-u}&=\int^{\infty}_0\sum^{\infty}_{h=1}e^{-hky+2\pi ifhk/l}e^{-hkx}x^{u-1}dx\\
&=\int^{\infty}_0\frac{x^{u-1}}{e^{k(x+y)-2\pi ifk/l}-1}dx.
\end{aligned}\end{align*}
So $B(u,v)$ can be extended meromorphically to the region $\Re{u}>0,\ \Re{v}>1,\ \Re{(u+v)}>2$ in the form
\begin{equation}\begin{aligned}B(u,v)=\frac{1}{\Gamma(u)\Gamma(v)}\sum_{k\le M}\sum_{l\le M}\frac{a(k)\overline{a(l)}}{l}\sum_{f=1}^l\int^{\infty}_0\frac{y^{v-1}}{e^{y-2\pi if/l}-1}\int^{\infty}_0\frac{x^{u-1}}{e^{k(x+y)-2\pi ifk/l}-1}dxdy.
\label{eq2}
\end{aligned}\end{equation}

Next, define
$$H(z;k,l,f)=\frac{1}{e^{kz-2\pi ifk/l}-1}-\frac{\delta(f)}{kz},
$$
where $\delta(f)=1$ or $0$ according to the condition whether $l$ divides $kf$ or not. The function $H(z;k,l,f)$ is holomorphic at $z=0$ and is $O(\min\{z^{-1},1\})$ for $z\ge0$.
So the inner integral on the right side of \eqref{eq2} can be written as
$$\frac{\delta(f)}{k}\int^{\infty}_0\frac{x^{u-1}}{x+y}dx+\int^{\infty}_0x^{u-1}H(x+y;k,l,f)dx,
$$
and since $$\int^{\infty}_0\frac{x^{u-1}}{x+y}dx=y^{u-1}\Gamma(u)\Gamma(1-u), \ \ \ \ \ 0<\Re{u}<1$$
we can write the right side of \eqref{eq2} by 
\begin{equation}
\frac{\Gamma(1-u)}{\Gamma(v)}\sum_{k\le M}\sum_{l\le M}\frac{a(k)\overline{a(l)}}{kl}\sum_{f=1}^l\delta(f)\int^{\infty}_0\frac{y^{u+v-2}}{e^{y-2\pi if/l}-a}dy+g(u,v;A)
\label{eq.5}
\end{equation}
in the region $0<\Re{u}<1,\ \Re{(u+v)}>2$, where
\begin{equation}\begin{aligned}
g(u,v;A)=&\frac{1}{\Gamma(u)\Gamma(v)}\sum_{k\le M}\sum_{l\le M}\frac{a(k)\overline{a(l)}}{l}\sum_{f=1}^l\int^{\infty}_0\frac{y^{v-1}}{e^{y-2\pi if/l}-1}\int^{\infty}_0x^{u-1}H(x+y;k,l,f)dx dy.
\label{eq3}
\end{aligned}\end{equation}
Similar to \eqref{eq1}, we see that the integral in the first term of \eqref{eq.5} equals
$$\Gamma(u+v-1)\sum^{\infty}_{m=1}e^{2\pi imf/l}m^{-u-v+1}.$$
Therefore, we obtain
$$B(u,v)=\frac{\Gamma(1-u)}{\Gamma(v)}\Gamma(u+v-1)\sum_{k\le M}\sum_{l\le M}\frac{a(k)\overline{a(l)}}{kl}\sum_{f=1}^l\delta(f)\phi\left(v+u-1,\frac{f}{l}\right)+g(u,v;A),$$
where $\phi(s,\alpha)=\sum_{m\ge 1}e^{2\pi\alpha im}m^{-s}$ is the Lerch zeta-function. Since
$$\sum^l_{f=1}\delta(f)\phi\left(u+v-1,\frac{f}{l}\right)=\sum^{(k,l)}_{j=1}\phi\left(u+v-1,\frac{j}{(k,l)}\right)=\frac{\zeta(u+v-1)}{(k,l)^{u+v-2}},$$
we obtain 
$$B(u,v)=\frac{\Gamma(1-u)}{\Gamma(v)}\Gamma(u+v-1)\zeta(u+v-1)\sum_{k\le M}\sum_{l\le M}\frac{a(k)\overline{a(l)}}{(k,l)^{u+v-1}[k,l]}+g(u,v;A)
$$
in the region $0<\Re{u}<1,\ \Re{(u+v)}>2$.

\section{The analytic continuation of $g(u,v;A)$}
Let $C$ be the contour which comes from $+\infty$ along the positive real axis, rounds the origin counterclockwise, and goes back to $+\infty$ again along the positive real axis. Then, for $0<\Re{u}<1,\ \Re{(u+v)}>2$, since
$$\int_Cx^{u-1}H(x+y;k,l,f)dx=\int^{ \varepsilon}_{\infty}+\int_{\circlearrowleft}+\int^{\infty}_{\varepsilon}=(e^{2\pi iu}-1)\int_{0}^{\infty}x^{u-1}H(x+y;k,l,f)dx$$
and we can do the same deformation for $y$, so we can write $g(u,v;A)$ as
\begin{equation}\begin{aligned}\frac{1}{\Gamma(u)\Gamma(v)(e^{2\pi iu}-1)(e^{2\pi iv}-1)}\sum_{k\le M}\sum_{l\le M}\frac{a(k)\overline{a(l)}}{l}\sum_{f=1}^l\int_C\frac{y^{v-1}}{e^{y-2\pi if/l}-1}\int_Cx^{u-1}H(x+y;k,l,f)dxdy,
\label{eq4}
\end{aligned}\end{equation}
where the right side double integral is convergent for $\Re{u}<1$ and any $v\in\mathbb{C}$. Hence \eqref{eq4} gives the meromorphic extension of $g(u,v;A)$ to the region, and therefore \eqref{eq3} is also valid in that region.

Therefore we obtain
\begin{align*}\begin{aligned}
\zeta(u)\zeta(v)&A(u)\overline{A(\bar{v})}=\zeta(u+v)\sum_{k\le M}\sum_{l\le M }\frac{a(k)\overline{a(l)}}{[k,l]^{u+v}}\\
&+\frac{\Gamma(1-u)}{\Gamma(v)}\Gamma(u+v-1)\zeta(u+v-1)\sum_{k\le M}\sum_{l\le M }\frac{a(k)\overline{a(l)}}{(k,l)^{u+v-1}[k,l]}\\
&+\frac{\Gamma(1-v)}{\Gamma(u)}\Gamma(u+v-1)\zeta(u+v-1)\sum_{k\le M}\sum_{l\le M }\frac{a(l)\overline{a(k)}}{(k,l)^{u+v-1}[k,l]}+g(u,v;A)+\overline{g(\bar{v},\bar{u};A)}.
\end{aligned}\end{align*}
Since changing $k$ and $l$ in the third double sum, we can combine the second and the third terms.

Now assume $0<\Re{u}<1$, and take $v=2\sigma-u$. We have
\begin{align*}\begin{aligned}
&\zeta(u)\zeta(2\sigma-u)A(u)\overline{A(\overline{2\sigma-u})}=\zeta(2\sigma)\sum_{k\le M}\sum_{l\le M }\frac{a(k)\overline{a(l)}}{[k,l]^{2\sigma}}\\
&+\left\{\frac{\Gamma(1-u)}{\Gamma(2\sigma-u)}+\frac{\Gamma(1-2\sigma+u)}{\Gamma(u)}\right\}\Gamma(2\sigma-1)\zeta(2\sigma-1)\sum_{k\le M}\sum_{l\le M }\frac{a(k)\overline{a(l)}}{(k,l)^{2\sigma-1}[k,l]}\\
&+g(u,2\sigma-u;A)+\overline{g(\overline{2\sigma-u},\bar{u};A)}
\end{aligned}\end{align*}
for $0<\Re{u}<1$.

Next, we consider another expression of $g(u,2\sigma-u;A)$ in the region $\Re{u}<0$.
Assume $R$ is a large positive integer, and $$C_R=C_R(k,l,f)=\left\{x=-y+\frac{2\pi if}{l}+\frac{2\pi}{k}(R+\frac{1}{2})e^{i\theta}|0\le\theta<2\pi\right\}.$$
Then we can easily see that $$H(x+y;k,l,f)=\frac{1}{e^{2\pi(R+1/2)e^{i\theta}}-1}-\frac{\delta(f)}{\frac{2\pi ifk}{l}+2\pi(R+1/2)e^{i\theta}}\ll1$$
if $x\in C_R(k,l,f)$. Hence
$$
\int_{C_R}H(x+y)x^{u-1}dx\ll\int^{2\pi}_0R^{\Re{u}-1}Rd\theta\ll R^{\Re{u}},
$$
which tends to 0 as $R\to\infty$ if $\Re{u}<0.$ Hence
\begin{equation}
\int_CH(x+y)x^{u-1}dx=-2\pi i \sum^*_n{Res_n(H(x+y)x^{u-1})}
\label{eq5}
\end{equation}
if $\Re{u}<0$, where $Res_n(*)$ denotes the residue of the function at $x=-y+2\pi i(f/l+n/k)$ and $\sum^*_n$ means the summation running over all integers $n\neq -kf/l$. Since
$$Res_{n}(H(x+y)x^{u-1})=\frac{1}{k}\left(-y+2\pi i\left(\frac{f}{l}+\frac{n}{k}\right)\right)^{u-1},$$
If we can interchange the summation and integration, we obtain
\begin{align*}\begin{aligned}
g(u,2\sigma-u;A)=\frac{-2\pi i}{\Gamma(u)\Gamma(2\sigma-u)(e^{2\pi iu}-1)(e^{2\pi i(2\sigma-u)-1})}
\sum_{k\le M}\sum_{l\le M}\frac{a(k)\overline{a(l)}}{kl}\sum_{f=1}^l\sum^*_nI(n;k,l,f),
\end{aligned}\end{align*}
where
\begin{align*}\begin{aligned}I(n;k,l,f)&=\int_C\frac{y^{2\sigma-u-1}}{e^{y-2\pi if/l}-1}\left(-y+2\pi i\left(\frac{f}{l}+\frac{n}{k}\right)\right)^{u-1}dy\\
&=(e^{2\pi i(2\sigma-1-u)}-1)\int^{\infty}_0\frac{y^{2\sigma-1-u}}{e^{y-2\pi if/l}-1}\left(-y+2\pi i\left(\frac{f}{l}+\frac{n}{k}\right)\right)^{u-1}dy.
\end{aligned}\end{align*}
In case $n>\frac{-kl}{f}$ we can write
$$-y+2\pi i\left(\frac{f}{l}+\frac{n}{k}\right)=e^{\pi i}y+2\pi e^{i\pi/2}\left(\frac{f}{l}+\frac{n}{k}\right)$$
and setting 
$$y=2\pi\left(\frac{f}{l}+\frac{n}{k}\right)e^{-\pi i/2}\eta,\ \ \ \arg\eta=\frac{\pi}{2} $$
we have
\begin{equation}\begin{aligned}
&I(n;k,l,f)=(e^{2\pi i(2\sigma-u)}-1)\int^{\infty}_0\frac{y^{2\sigma-1-u}}{e^{y-2\pi if/l}-1}\left(-y+2\pi i\left(\frac{f}{l}+\frac{n}{k}\right)\right)^{u-1}dy\\
&=-(e^{2\pi i(2\sigma-u)}-1)\int^{\infty}_0\frac{(2\pi\left(\frac{f}{l}+\frac{n}{k}\right)e^{-\pi i/2}\eta)^{2\sigma-1-u}}{\exp{(-2\pi i((\frac{f}{l}+\frac{n}{l})\eta+\frac{f}{l}))-1} }\left[2\pi i\left(\frac{f}{l}+\frac{n}{k}\right)\right]^{u}(1+\eta)^{u-1}d\eta\\
&=(1-e^{2\pi i(2\sigma-u)})(e^{\pi i/2})^{-2\sigma+1+2u}\left[2\pi \left(\frac{f}{l}+\frac{n}{k}\right)\right]^{2\sigma-1}\int^{i\infty}_0\frac{\eta^{2\sigma-1-u}(1+\eta)^{u-1}d\eta}{\exp{(-2\pi i((\frac{f}{l}+\frac{n}{l})\eta+\frac{f}{l}))} -1}\\
%&=(e^{\pi i(2\sigma-u)}-e^{-\pi i(2\sigma-u)}(e^{\pi i})^{\sigma+3/2}\left[2\pi \left(\frac{f}{l}+\frac{n}{k}\right)\right]^{2\sigma-1}\int^{i\infty}_0\frac{\eta^{2\sigma-1-u}(1+\eta)^{u-1}d\eta}{\exp{(-2\pi i((\frac{f}{l}+\frac{n}{l})\eta+\frac{f}{l}))}-1 }\\
&=(e^{\pi i(2\sigma-u)}-e^{-\pi i(2\sigma-u)}(e^{\pi i})^{\sigma+3/2}\left[2\pi \left(\frac{f}{l}+\frac{n}{k}\right)\right]^{2\sigma-1}\\
&\times\sum_{m=1}^{\infty}\int^{i\infty}_0\eta^{2\sigma-1-u}(1+\eta)^{u-1}\exp{\left(2\pi im\left(\left(\frac{f}{l}+\frac{n}{l}\right)\eta+\frac{f}{l}\right)\right)}d\eta,
\label{eq6}
\end{aligned}\end{equation}
where the second equality is valid if the summation and the integration can be interchanged.

In case $n<kf/l$, since 
$$-y+2\pi i\left(\frac{f}{l}+\frac{n}{k}\right)=e^{\pi i}y+2\pi e^{3i\pi/2}\left|\frac{f}{l}+\frac{n}{k}\right|,$$
setting
$$y=2\pi\left|\frac{f}{l}+\frac{n}{k}\right|e^{-\pi i/2}\eta,\ \ \ \arg\eta=-\frac{\pi}{2} ,$$
we find that $I(n;k,l,f)$ has an expression similar to that in \eqref{eq6} with the integral over the interval $[0,-i\infty)$, instead of $[0,i\infty)$. Substituting these results into \eqref{eq5}, we obtain
\begin{align*}\begin{aligned}
g(u,2\sigma-u;A)=&\frac{-i(e^{\pi i(2\sigma-u)}-e^{-\pi i(2\sigma-u)})(e^{\pi i})^{\sigma+3/2}\left[2\pi \right]^{2\sigma}}{\Gamma(u)\Gamma(2\sigma-u)(e^{2\pi iu}-1)(e^{2\pi i(2\sigma-u)-1})}\sum_{k\le M}\sum_{l\le M}\frac{a(k)\overline{a(l)}}{kl}\left(J_+(k,l)+J_-(k,l)\right),
\end{aligned}\end{align*}
where\begin{align*}\begin{aligned}
J_{\pm}(k,l)&=\sum_{f=1}^l\sum_n^*\left(\frac{f}{l}+\frac{n}{k}\right)^{2\sigma-1}\sum_{m=1}^{\infty}\int_0^{\pm i\infty}\eta^{2\sigma-1-u}(1+\eta)^{u-1}\exp{\left(2\pi im\left(\left(\frac{f}{l}+\frac{n}{l}\right)\eta+\frac{f}{l}\right)\right)}d\eta
\end{aligned}\end{align*}
and the summation with respect to $n$ runs over all $n>-kf/l$ (resp. all $n<-kf/l$) for $J_+$ (resp. $J_-$).

Using the notations $\kappa=\frac{k}{(k,l)}$ and $\lambda=\frac{l}{(k,l)}$, we have
\begin{align*}\begin{aligned}
\exp{\left(2\pi im\left(\left(\frac{f}{l}+\frac{n}{l}\right)\eta+\frac{f}{l}\right)\right)}=\exp{\left(2\pi i\frac{mf}{l}\right)}\exp{\left(2\pi i\frac{m(f\kappa+n\lambda)}{(k,l)\kappa\lambda}\eta\right)}.
\end{aligned}\end{align*}
Put $h=f\kappa+n\lambda$. If $n>-kf/l$, then $h>0$. On the other hand, for any positive integer $h$, we can find integers $f$ and $n$ such that $1\le f\le l$ and $h=f\kappa+n\lambda$. In fact, since $(\kappa,\lambda)=1$ we can find integers $x,\ y$ with $x\kappa+y\lambda=1$. Then $f=hx+v\lambda$ and $n=hy-v\kappa$ for any integer $v$ satisfy $h=f\kappa+n\lambda$. Choosing $v$ suitably we have $1\le f\le l$, as desired. Therefore
\begin{align*}\begin{aligned}
J_{+}(k,l)=\sum_{m=1}^{\infty}\sum_{h=1}^{\infty}\left(\frac{h}{\kappa\lambda(k,l)}\right)^{2\sigma-1}\sum_f\exp{\left(2\pi i\frac{mf}{l}\right)}\int_0^{\pm i\infty}\eta^{2\sigma-1-u}(1+\eta)^{u-1}\exp{\left(2\pi i\frac{mh}{(k,l)\kappa\lambda}\eta\right)}d\eta,
\end{aligned}\end{align*}
where the innermost sun runs over all $f$ satisfying $1\le f\le l$ and there exists an integer $n$ with $h=f\kappa+n\lambda$.
Recall that $\bar{\kappa}$ is an integer satisfying $\kappa\bar{\kappa}\equiv1(\mod\lambda)$. Then $h=f\kappa+n\lambda$ implies $f\equiv\bar{\kappa}(\mod \lambda)$. The number of $f$ satisfying this congruence condition and $1\le  f\le l$ is $l/\lambda=(k,l)$. Hence
\begin{align*}\begin{aligned}
\sum_f\exp{\left(2\pi i\frac{mf}{l}\right)}=\sum^{(k,l)}_{j=1}\exp{\left(2\pi i\frac{m}{l}(h\bar{\kappa}+j\lambda)\right)}
=\exp{\left(2\pi i\frac{mh\bar{\kappa}}{l}\right)}\sum^{(k,l)}_{j=1}\exp{\left(2\pi i\frac{mj}{(k,l)}\right)},
\end{aligned}\end{align*}
and the inner sum equals $(k,l)$ if $(k,l)|m$ and 0 otherwise. Therefore, setting $m=(k,l)\mu$, we have
$$J_+(k,l)=(k,l)\sum_{\mu=1}^{\infty}\sum_{h=1}^{\infty}\left(\frac{h}{k\lambda(k,l)}\right)^{2\sigma-1}\exp{\left(2\pi i\frac{\mu h\bar{\kappa}}{\lambda}\right)}\int^{i\infty}_0\eta^{2\sigma-1-u}(1+\eta)^{u-1}\exp{\left(2\pi i\frac{\mu h}{\kappa\lambda}\eta\right)}d\eta,
$$
and setting further $\mu h=n$, we obtain
\begin{equation}\begin{aligned}
J_+(k,l)&=\frac{(k,l)}{\left(\kappa\lambda(k,l)\right)^{2\sigma-1}}\sum_{n=1}^{\infty}\left(\sum_{h|n}h^{2\sigma-1}\right)\exp{\left(2\pi i\frac{n\bar{\kappa}}{\lambda}\right)}\int^{i\infty}_0\eta^{2\sigma-1-u}(1+\eta)^{u-1}e^{2\pi in\eta/\kappa\lambda}d\eta\\
&=\frac{(k,l)}{\left(\kappa\lambda(k,l)\right)^{2\sigma-1}}\sum_{n=1}^{\infty}\sigma_{2\sigma-1}(n)\exp{\left(2\pi i\frac{n\bar{\kappa}}{\lambda}\right)}\int^{i\infty}_0\eta^{2\sigma-1-u}(1+\eta)^{u-1}e^{2\pi in\eta/\kappa\lambda}d\eta.
\label{eq7}
\end{aligned}\end{equation}
Similarly,
\begin{equation}\begin{aligned}\label{eq8}
J_-(k,l)&=\frac{(k,l)}{\left(\kappa\lambda(k,l)\right)^{2\sigma-1}}\sum_{n=1}^{\infty}\sigma_{2\sigma-1}(n)\exp{\left(-2\pi i\frac{n\bar{\kappa}}{\lambda}\right)}\int^{i\infty}_0\eta^{2\sigma-1-u}(1+\eta)^{u-1}e^{-2\pi in\eta/\kappa\lambda}d\eta.
\end{aligned}\end{equation}
The integrals on the right side of \eqref{eq7} and \eqref{eq8} are estimated as $O(n^{\Re{u}-2\sigma})$, so the infinite series including those integrals are absolutely convergent for $\Re{u}<0$. Hence we can now justify the above interchanges of summation and integration.

The next thing is to change the integral above from $[0,i\infty)$ to $[0,\infty)$. This will allow the application of the integral over $[0,\infty)$ using Atkinson's Lemma.

So above all, we find 
\begin{equation}\begin{aligned}\label{eq:g}
g(u,2\sigma-u;A)=&\frac{-2\pi i(e^{\pi i(2\sigma-u)}-e^{-\pi i(2\sigma-u)}(e^{\pi i})^{\sigma+3/2}\left[2\pi \right]^{2\sigma-1}}{\Gamma(u)\Gamma(2\sigma-u)(e^{2\pi iu}-1)(e^{2\pi i(2\sigma-u)-1})}\\
&\times\sum_{k\le M}\sum_{l\le M}\frac{a(k)\overline{a(l)}}{[k,l]^{2\sigma}}\sum_{n\neq0}\sigma_{2\sigma-1}(|n|)\exp{\left(2\pi i\frac{n\bar{\kappa}}{\lambda}\right)}h(u,\frac{n}{\kappa\lambda})
\end{aligned}\end{equation}
for $\Re{u}<0,$ where 
\begin{equation*}
h(u,x)=\int^{\infty}_0y^{2\sigma-1-u}(1+y)^{u-1}\exp{(2\pi ixy)}dy.
\end{equation*}

%%%%%%%%%%%%%%%%%%%%%%%%%%%%%%%%%%%%%%%%%%%%%

Since $h(u,x)=O(x^{\Re{u}-2\sigma})$, by using the \eqref{eq:D}, we can find $h(u,x/\kappa\lambda)D_{2\sigma-1}(x;\bar{\kappa}/\lambda)=O(x^{\Re{u}})$, when $\Re{u}<0$, and $h(u,\frac{x}{\kappa\lambda})D_{2\sigma-1}(x,\frac{\bar{\kappa}}{\lambda})\to0$ as $x\to\infty$. Therefore by integration by parts we have
\begin{equation}\begin{aligned}
&\sum_{n>x}\sigma_{2\sigma-1}(n)e\left(\frac{\bar{\kappa}n}{\lambda}\right)h\left(u,\frac{n}{\kappa\lambda}\right)\\
&=\int^{\infty}_xh\left(u,\frac{n}{\kappa\lambda}\right)dD_{2\sigma-1}\left(x,\frac{\bar{\kappa}}{\lambda}\right)\\
&=-h\left(u,\frac{x}{\kappa\lambda}\right)D_{2\sigma-1}\left(x,\frac{\bar{\kappa}}{\lambda}\right)-\int_x^{\infty}\frac{\partial{h\left(u,\frac{x}{\kappa\lambda}\right)}}{\partial{x}}D_{2\sigma-1}\left(x,\frac{\bar{\kappa}}{\lambda}\right)dx.
\label{eq:deh}
\end{aligned}\end{equation}
Substituting \eqref{eq:deh} into the integral on the right side of \eqref{eq:g}, and applying integration by parts once more, we obtain
$$\sum_{n=1}^{\infty}\sigma_{2\sigma-1}(n)e\left(\frac{\bar{\kappa}n}{\lambda}\right)h\left(u,\frac{n}{\kappa\lambda}\right)=g_1(u)-g_2(u)+g_3(u)-g_4(u)
$$
for $\Re u<0$, where
\begin{equation*}\begin{aligned}
&g_1(u)=\sum_{n\le X}\sigma_{2\sigma-1}(n)e\left(\frac{\bar{\kappa}n}{\lambda}\right)h\left(u,\frac{n}{\kappa\lambda}\right),\\
&g_2(u)=h\left(u,\frac{X}{\kappa\lambda}\right)\Delta_{2\sigma-1}\left(X,\frac{\bar{\kappa}}{\lambda}\right),\\
&g_3(u)=\int^{\infty}_X\left(\frac{\zeta(2-2\sigma)}{\lambda^{2-2\sigma}}+\frac{\zeta(2\sigma)}{\lambda^{2\sigma}}x^{2\sigma-q}\right)h\left(u,\frac{x}{\kappa\lambda}\right)dx,\ \ \\
&g_4(u)=\int_X^{\infty}\frac{\partial{h\left(u,\frac{x}{\kappa\lambda}\right)}}{\partial{x}}\Delta_{2\sigma-1}\left(x,\frac{\bar{\kappa}}{\lambda}\right)dx.
\end{aligned}\end{equation*}
Similarly, we have 
$$\sum_{n=1}^{\infty}\sigma_{2\sigma-1}(n)e\left(-\frac{\bar{\kappa}n}{\lambda}\right)h\left(u,\frac{-n}{\kappa\lambda}\right)=\overline{g_1(\bar{u})-g_2(\bar{u})+g_3(\bar{u})-g_4(\bar{u})}$$
hence from \eqref{eq:g} we have
\begin{equation}\begin{aligned}g(u,2\sigma-u;A)=&\frac{-2\pi i(e^{\pi i(2\sigma-1-u)}-e^{-\pi i(2\sigma-1-u)}(e^{\pi i})^{\sigma+1/2}\left[2\pi \right]^{2\sigma-1}}{\Gamma(u)\Gamma(2\sigma-u)(e^{2\pi iu}-1)(e^{2\pi i(2\sigma-u)-1})}\\
&\times\sum_{k\le M}\sum_{l\le M}\frac{a(k)\overline{a(l)}}{[k,l]^{2\sigma}}\left(g_1(u)-g_2(u)+g_3(u)-g_4(u)\right.\\
&\left.+\overline{g_1(\bar{u})-g_2(\bar{u})+g_3(\bar{u})-g_4(\bar{u})}\right)
\label{eq:gg}
\end{aligned}\end{equation}
for $\Re u<0$.
In view of $h(u,x)=O(x^{\Re u-1})$, $g_1(u)$ and $g_2(u)$ are holomorphic in the region $\Re u<2\sigma$. Also, we have $$\frac{\partial{h\left(u,\frac{x}{\kappa\lambda}\right)}}{\partial{x}}=O(x^{\Re u-2\sigma-1}),$$
and
 $$\frac{\partial{h\left(u,\frac{x}{\kappa\lambda}\right)}}{\partial{x}}\Delta_{2\sigma-1}\left(x,\frac{\bar{\kappa}}{\lambda}\right)=O(x^{\Re u-2\sigma-1+2\sigma-1+1/(5-4\sigma)+\varepsilon})=O(x^{\Re u-2+1/(5-4\sigma)+\varepsilon}).$$
These estimates imply that $g_4(u)$ is holomorphic for $\Re{u}<1-\frac{1}{5-4\sigma}$. 

As for $g_3(u)$, 
\begin{align*}\begin{aligned}
&\int^{\infty}_X\left(\frac{\zeta(1-a)}{\lambda^{1-a}}+\frac{\zeta(1+a)}{\lambda^{1+a}}x^{a}\right)\int^{\infty}_0y^{2\sigma-1-u}(1+y)^{u-1}\exp{(2\pi i\frac{xy}{\kappa\lambda})}dydx\\
&=\left[\left(\frac{\zeta(1-a)}{\lambda^{1-a}}+\frac{\zeta(1+a)}{\lambda^{1+a}}x^{a}\right)\int^{\infty}_0y^{2\sigma-1-u}(1+y)^{u-1}\frac{\exp{(2\pi i\frac{xy}{\kappa\lambda})}}{2\pi i\frac{y}{\kappa\lambda}}dy\right]^{\infty}_X\\
&-\int_X^{\infty}\frac{\zeta(2\sigma)}{\lambda^{2\sigma}}x^{2\sigma-2}dx\int_0^{\infty}y^{-u+2\sigma-2}(1+y)^{u-1}\frac{e^{2\pi i\frac{xy}{\kappa\lambda}}}{2\pi i/\kappa\lambda}dy\\
&=\frac{\zeta(2-2\sigma)\lambda^{2\sigma-2}+\frac{\zeta(2\sigma)}{\lambda^{2\sigma}}x^{2\sigma-1}}{-2\pi i}\kappa\lambda\int_0^{\infty}y^{-u+2\sigma-2}(1+y)^{u-1}e^{2\pi i\frac{xy}{\kappa\lambda}}dy\\
&-\int_X^{\infty}\frac{\zeta(2\sigma)}{\lambda^{2\sigma}}x^{2\sigma-2}dx\int_0^{\infty}y^{-u+2\sigma-2}(1+y)^{u-1}\frac{e^{2\pi i\frac{xy}{\kappa\lambda}}}{2\pi i/\kappa\lambda}dy.
\end{aligned}\end{align*}
the last equality is verified by the computation as
\begin{align*}\begin{aligned}
&-\int_X^{\infty}\frac{\zeta(2\sigma)}{\lambda^{2\sigma}}x^{2\sigma-2}dx\int_0^{\infty}y^{-u+2\sigma-2}(1+y)^{u-1}\frac{e^{2\pi i\frac{xy}{\kappa\lambda}}}{2\pi i/\kappa\lambda}dy\\
&=\frac{\kappa\zeta(2\sigma)(2\sigma-1)}{2\pi i\lambda^{2\sigma-1}}\int_X^{\infty}dx\int_0^{\infty}y^{-y+2\sigma-2}(x+y)^{u-1}e^{2\pi iy/\kappa\lambda}dy\\
&=\frac{\kappa\zeta(2\sigma)(2\sigma-1)}{2\pi i\lambda^{2\sigma-1}}u^{-1}\int_0^{\infty}y^{-y+2\sigma-2}\left[(x+y)^{u}\right]^{\infty}_Xe^{2\pi iy/\kappa\lambda}dy\\
&=\frac{\kappa\zeta(2\sigma)(2\sigma-1)}{2\pi i\lambda^{2\sigma-1}}\frac{X^{2\sigma-1}}{-u}\int_0^{\infty}y^{-u+2\sigma-2}(1+y)^ue^{2\pi i\frac{Xy}{\kappa\lambda}}dy.
\end{aligned}\end{align*}
Therefore, we obtain
\begin{equation}\begin{aligned}
g_3(u)+\overline{g_3(\bar{u})}&=\frac{\zeta(2-2\sigma)\lambda^{2\sigma-2}+\frac{\zeta(2\sigma)}{\lambda^{2\sigma}}X^{2\sigma-1}}{-\pi }\kappa\lambda\int_0^{\infty}y^{-u+2\sigma-2}(1+y)^{u-1}\sin{\left(2\pi \frac{Xy}{\kappa\lambda}\right)}dy\\
&+\frac{\kappa\zeta(2\sigma)(2\sigma-1)X^{2\sigma-1}}{\pi \lambda^{2\sigma-1}u}\int_0^{\infty}y^{-u+2\sigma-2}(1+y)^u\sin{\left(2\pi \frac{Xy}{\kappa\lambda}\right)}dy.
\label{eq:g3}
\end{aligned}\end{equation}
The two integrals on the right side are convergent uniformly in $\Re{u}<2-2\sigma$. Hence \eqref{eq:gg} and \eqref{eq:g3} gives an expression of $g(u,2\sigma-u;A)$ valid for $\Re{u}<\min{\left((1-\frac{1}{5-4\sigma}),2-2\sigma\right)}$. By the same argument we have and expression of $\overline{g(\overline{2\sigma-u},\bar{u};A)}$ which is valid for $\Re{u}>\max{\left((2\sigma-1+\frac{1}{5-4\sigma}),4\sigma-2\right)}$.

\section{The mean square and Atkinson's lemma}
Now let us consider the mean square of the $\zeta(\sigma+it)A(\sigma+it)$, letting $C_1$, $C_2$, and $C^{\star}$ be the same as in the statement of Theorem\ref{thm1}. Assume  $C_1T<Y<C_2T$, $T\ge C^{\star}$, and let $X=X(\kappa,\lambda)=\kappa\lambda Y$.
We have
\begin{equation*}\begin{aligned}
&\zeta(u)\zeta(2\sigma-u)A(u)\overline{A(\overline{2\sigma-u})}=\zeta(2\sigma)\sum_{k\le M}\sum_{l\le M }\frac{a(k)\overline{a(l)}}{[k,l]^{2\sigma}}\\
&+\left\{\frac{\Gamma(1-u)}{\Gamma(2\sigma-u)}+\frac{\Gamma(1-2\sigma+u)}{\Gamma(u)}\right\}\Gamma(2\sigma-1)\zeta(2\sigma-1)\sum_{k\le M}\sum_{l\le M }\frac{a(k)\overline{a(l)}}{(k,l)^{2\sigma-1}[k,l]}\\
&+g(u,2\sigma-u;A)+\overline{g(\overline{2\sigma-u},\bar{u};A)}
\end{aligned}\end{equation*}
and
\begin{equation*}\begin{aligned}&g(u,2\sigma-u;A)=\frac{-2\pi i(e^{\pi i(2\sigma-1-u)}-e^{-\pi i(2\sigma-1-u)}(e^{\pi i})^{\sigma+1/2}\left[2\pi \right]^{2\sigma-1}}{\Gamma(u)\Gamma(2\sigma-u)(e^{2\pi iu}-1)(e^{2\pi i(2\sigma-u)-1})}\\
&\times \sum_{k\le M}\sum_{l\le M}\frac{a(k)\overline{a(l)}}{[k,l]^{2\sigma}}\left(g_1(u)-g_2(u)+g_3(u)-g_4(u)+\right.\left.\overline{g_1(\bar{u})-g_2(\bar{u})+g_3(\bar{u})-g_4(\bar{u})}\right),
\end{aligned}\end{equation*}
where
\begin{align*}
g_1(u)=\sum_{n\le X}\sigma_{2\sigma-1}(n)e\left(\frac{\bar{\kappa}n}{\lambda}\right)h\left(u,\frac{n}{\kappa\lambda}\right),
\end{align*}
\begin{align*}
&g_2(u)=h\left(u,\frac{X}{\kappa\lambda}\right)\Delta_{2\sigma-1}\left(X,\frac{\bar{\kappa}}{\lambda}\right),\end{align*}
\begin{align*}&g_3(u)+\overline{g_3(\bar{u})}=\frac{\zeta(2-2\sigma)\lambda^{2\sigma-2}+\frac{\zeta(2\sigma)}{\lambda^{2\sigma}}X^{2\sigma-1}}{-\pi }\kappa\lambda\int_0^{\infty}y^{-u+2\sigma-2}(1+y)^{u-1}\sin{\left(2\pi \frac{Xy}{\kappa\lambda}\right)}dy\\
&\ \ \ \ \ \ \ \ \ \ \ \ \ \ \ +\frac{\kappa\zeta(2\sigma)(2\sigma-1)X^{2\sigma-1}}{\pi \lambda^{2\sigma-1}u}\int_0^{\infty}y^{-u+2\sigma-2}(1+y)^u\sin{\left(2\pi \frac{Xy}{\kappa\lambda}\right)}dy,\end{align*}
\begin{align*}g_4(u)=\int_X^{\infty}\frac{\partial{h\left(u,\frac{x}{\kappa\lambda}\right)}}{\partial{x}}\Delta_{2\sigma-1}\left(x,\frac{\bar{\kappa}}{\lambda}\right)dx,\end{align*}
\begin{align*}h(u,x)=\int^{\infty}_0y^{2\sigma-1-u}(1+y)^{u-1}\exp{(2\pi ixy)}dy.
\end{align*}
Let $u=\sigma+it$, and integrate the both sides on the interval $[T,2T]$. 

By Stirling's formula, we obtain
\begin{align*}\begin{aligned}
\log{\Gamma(\sigma+it)}=%&(\sigma+it-\frac{1}{2})\log(\sigma+it)-(\sigma+it)+\log\sqrt{2\pi}+\frac{1}{12(\sigma+it)}+O(t^{-3})\\
%&=\left(\sigma+it-\frac{1}{2}\right)\log{(it)}+(\sigma+it-\frac{1}{2})\left(\frac{\sigma}{it}-\frac{1}{2}\left(\frac{\sigma}{it}\right)^2+O(t^{-3})\right)\\
%&-\sigma-it+\log(\sqrt{2\pi})+\frac{1}{12}\frac{1}{\sigma+it}+O(t^{-3})\\
=\left(\sigma+it-\frac{1}{2}\right)\log{(it)}+\frac{\sigma^2-\sigma}{2it}-it+\frac{1}{12}\frac{1}{\sigma+it}+\log(\sqrt{2\pi})+O(t^{-2}),
\end{aligned}\end{align*}
then
\begin{align*}\begin{aligned}
\log\left(\frac{\Gamma(1-\sigma+it)}{\Gamma(\sigma+it)}\right)=(1-2\sigma)\log(it)+\frac{1}{12}\left(\frac{1}{1-\sigma+it}-\frac{1}{\sigma+it}\right)+O(t^{-2}).
\end{aligned}\end{align*}
Since
\begin{align*}\begin{aligned}
\left(\frac{1}{1-\sigma+it}-\frac{1}{\sigma+it}\right)=\frac{1}{it}\left(\frac{1}{\frac{1-\sigma}{it}+1}-\frac{1}{\frac{\sigma}{it}+1}\right)=-(1-2\sigma)t^{-2}+O(t^{-3})\ 
\end{aligned}\end{align*}
for $\left|\frac{\sigma}{it}\right|<1$,
then
\begin{align*}\begin{aligned}
\log\left(\frac{\Gamma(1-\sigma+it)}{\Gamma(\sigma+it)}\right)&=(1-2\sigma)\log(it)+O(t^{-2}).
\end{aligned}\end{align*}
So we have
\begin{align*}\begin{aligned}
\frac{\Gamma(1-\sigma+it)}{\Gamma(\sigma+it)}&=e^{(1-2\sigma)\log(it)+O(t^{-2})}=(it)^{(1-2\sigma)}(1+O(t^{-2}))=e^{(1-2\sigma)\frac{\pi i}{2}}t^{1-2\sigma}(1+O(t^{-2}));
\end{aligned}\end{align*}
and it is conjugate
\begin{equation*}\begin{aligned}
\frac{\Gamma(1-\sigma-it)}{\Gamma(\sigma-it)}&=e^{-(1-2\sigma)\frac{\pi i}{2}}t^{1-2\sigma}(1+O(t^{-2})).
\end{aligned}\end{equation*}
Hence we obtain
\begin{align*}\begin{aligned}
\int_T^{2T}\left\{\frac{\Gamma(1-\sigma-it)}{\Gamma(\sigma-it)}+\frac{\Gamma(1-\sigma+it)}{\Gamma(\sigma+it)}\right\}=\frac{\cos(\frac{2\sigma-1}{2}\pi)}{1-\sigma}(2^{1-\sigma}-1)T^{2-2\sigma}+O(1).
\end{aligned}\end{align*}
Noting when $u=\sigma+it$, $\frac{1}{\Gamma(u)\Gamma(2\sigma-u)\sin(\pi u)}=\frac{\Gamma(1-\sigma-it)}{\Gamma(\sigma-it)\pi}$, and we have $$\left(\frac{1}{\Gamma(\overline{2\sigma-\overline{2\sigma-u}})\Gamma(2\sigma-u)\sin(\pi (\overline{2\sigma-u}))}\right)=\left(\frac{\Gamma(1-\sigma-it)}{\Gamma(\sigma-it)\pi}\right).$$
So, we can write $g(u,2\sigma-u;A)+\overline{g(\overline{2\sigma-u},\bar{u};A)}$ as 
\begin{align*}\begin{aligned}
2\Re\left\{-(2\pi)^{2\sigma-1}e^{(1/2-3\sigma)\pi i}{\left\{\frac{\Gamma(1-\sigma-it)}{\Gamma(\sigma-it)}\sum_{i=1}^4(-1)^{i-1}(g_i(\sigma+it)+\overline{g_i(\sigma-it)})\right\}}\right\},
\end{aligned}\end{align*}
we write $$\int_T^{2T}g(u,2\sigma-u;A)+\overline{g(\overline{2\sigma-u},\bar{u};A)}dt=I_1-I_2+I_3-I_4,$$where
\begin{align*}
I_v=\sum_{k\le M}\sum_{l\le M}\frac{2}{[k,l]^{2\sigma}\pi}\Re\left\{Ca(k)\overline{a(l)}\int_T^{2T}{\left\{\frac{\Gamma(1-\sigma-it)}{\Gamma(\sigma-it)}(g_v(\sigma+it)+\overline{g_v(\sigma-it)})\right\}}dt\right\}
\end{align*}
for $v=1,2,3,4$ and $C=-(2\pi)^{2\sigma-1}e^{(1/2-3\sigma)\pi i}$.

Next, we calculate $\int_T^{2T}{\left\{\frac{\Gamma(1-\sigma-it)}{\Gamma(\sigma-it)}(g_v(\sigma+it)+\overline{g_v(\sigma-it)})\right\}}dt.$

For $v=1$, the double integral is absolutely convergent, and hence, by integration by parts, we have:
\begin{align*}\begin{aligned}
&\int_T^{2T}{\frac{\Gamma(1-\sigma-it)}{\Gamma(\sigma-it)}g_1(\sigma+it)}dt\\
&=\int_T^{2T}e^{-\pi i \frac{1-2\sigma}{2}}t^{1-2\sigma}(1+O(t^{-2}))\sum_{n\le X}\sigma_{2\sigma-1}(n)e\left(\frac{\bar{\kappa}n}{\lambda}\right)\int_0^{\infty}\frac{e(\frac{n}{\kappa\lambda}y)}{y^{1-\sigma}(1+y)^{1-\sigma}}e^{it(\log(\frac{1+y}{y}))}dydt\\
&=e^{-\pi i \frac{1-2\sigma}{2}}\sum_{n\le X}\sigma_{2\sigma-1}(n)e\left(\frac{\bar{\kappa}n}{\lambda}\right)\int_0^{\infty}\frac{e(\frac{n}{\kappa\lambda}y)}{y^{1-\sigma}(1+y)^{1-\sigma}}dy\int_T^{2T}e^{it\log{\frac{1+y}{y}}}t^{1-2\sigma}dt\\
&+e^{-\pi i \frac{1-2\sigma}{2}}\sum_{n\le X}\sigma_{2\sigma-1}(n)e\left(\frac{\bar{\kappa}n}{\lambda}\right)\int_0^{\infty}\frac{e(\frac{n}{\kappa\lambda}y)}{y^{1-\sigma}(1+y)^{1-\sigma}}dy\int_T^{2T}e^{it\log{\frac{1+y}{y}}}O(t^{-1-2\sigma})dt\\
&=e^{-\pi i \frac{1-2\sigma}{2}}\sum_{n\le X}\sigma_{2\sigma-1}(n)e\left(\frac{\bar{\kappa}n}{\lambda}\right)\int_0^{\infty}\frac{e(\frac{n}{\kappa\lambda}y)}{y^{1-\sigma}(1+y)^{1-\sigma}}dy\left\{\left[t^{1-2\sigma}\frac{e^{it\log\frac{1+y}{y}}}{i\log\frac{1+y}{y}}\right]^{2T}_T\right\}\\
&-e^{-\pi i \frac{1-2\sigma}{2}}\sum_{n\le X}\sigma_{2\sigma-1}(n)e\left(\frac{\bar{\kappa}n}{\lambda}\right)\int_T^{2T}(1-2\sigma)t^{-2\sigma}dt\int_0^{\infty}\frac{e(\frac{n}{\kappa\lambda}y)}{y^{1-\sigma}(1+y)^{1-\sigma}}\frac{e^{it\log(y/(1+y))}}{i\log(y/(1+y))}dy\\
&+e^{-\pi i \frac{1-2\sigma}{2}}\sum_{n\le X}\sigma_{2\sigma-1}(n)e\left(\frac{\bar{\kappa}n}{\lambda}\right)\int_T^{2T}O(t^{-1-2\sigma})dt\int_0^{\infty}\frac{e(\frac{n}{\kappa\lambda}y)}{y^{1-\sigma}(1+y)^{1-\sigma}}e^{it\log{\frac{1+y}{y}}}dy\\
&=I_{1M}-I_{r1}+I_{r2},
\end{aligned}\end{align*}
say, so changing the order of integration, we get
\begin{equation}\begin{aligned}
I_1&=\sum_{k\le M}\sum_{l\le M}\sum_{n\le X}(S_{1M}(t;n,k,l)|^{2T}_{t=T}-S_{r1}+S_{r2}),
\label{eq:SumI1}
\end{aligned}\end{equation}
\begin{equation*}\begin{aligned}
I_2&=\sum_{k\le M}\sum_{l\le M}\sum_{n\le X}(S_2(t;n,k,l)|^{2T}_{t=T}+S_{2r}),
%\label{eq:SumI2}
\end{aligned}\end{equation*}
where
\begin{align*}\begin{aligned}
&S_{1M}(t;n,k,l)\\
&=2(2\pi)^{2\sigma-1}\Im\left\{e^{-2\sigma\pi i}\frac{a(k)\overline{a(l)}}{[k,l]^{2\sigma}}\sigma_{2\sigma-1}(n)e\left(\frac{\bar{\kappa}n}{\lambda}\right)t^{1-2\sigma}\int_0^{\infty}\frac{e(\frac{n}{\kappa\lambda}y)e^{it\log\frac{1+y}{y}}}{y^{1-\sigma}(1+y)^{1-\sigma}\log\frac{1+y}{y}}dy\right\}\\
&+2(2\pi)^{2\sigma-1}\Im\left\{e^{-2\sigma\pi i}\frac{a(k)\overline{a(l)}}{[k,l]^{2\sigma}}\sigma_{2\sigma-1}(n)e\left(-\frac{\bar{\kappa}n}{\lambda}\right)t^{1-2\sigma}\int_0^{\infty}\frac{e(-\frac{n}{\kappa\lambda}y)e^{it\log\frac{1+y}{y}}}{y^{1-\sigma}(1+y)^{1-\sigma}\log\frac{1+y}{y}}dy\right\}\\
&=S_{11}(t;n,k,l)+S_{12}(t;n,k,l);
\end{aligned}\end{align*}
\begin{align*}\begin{aligned}
&S_{r1}(n,k,l)\\
&=2(1-2\sigma)(2\pi)^{2\sigma-1}\Im\left\{e^{-2\sigma\pi i}\frac{a(k)\overline{a(l)}}{[k,l]^{2\sigma}}\sigma_{2\sigma-1}(n)e\left(\frac{\bar{\kappa}n}{\lambda}\right)\int_T^{2T}t^{-2\sigma}dt\int_0^{\infty}\frac{e(\frac{n}{\kappa\lambda}y)e^{it\log\frac{1+y}{y}}}{y^{1-\sigma}(1+y)^{1-\sigma}\log\frac{1+y}{y}}dy\right\}\\
&+2(1-2\sigma)(2\pi)^{2\sigma-1}\Im\left\{e^{-2\sigma\pi i}\frac{a(k)\overline{a(l)}}{[k,l]^{2\sigma}}\sigma_{2\sigma-1}(n)e\left(-\frac{\bar{\kappa}n}{\lambda}\right)\int_T^{2T}t^{-2\sigma}dt\int_0^{\infty}\frac{e(-\frac{n}{\kappa\lambda}y)e^{it\log\frac{1+y}{y}}}{y^{1-\sigma}(1+y)^{1-\sigma}\log\frac{1+y}{y}}dy\right\};
\end{aligned}\end{align*}
\begin{equation}\begin{aligned}
&S_{r2}(n,k,l)\\
&=2(2\pi)^{2\sigma-1}\Im\left\{e^{-2\sigma\pi i}\frac{a(k)\overline{a(l)}}{[k,l]^{2\sigma}}\sigma_{2\sigma-1}(n)e\left(\frac{\bar{\kappa}n}{\lambda}\right)\int_T^{2T}O(t^{-1-2\sigma})dt\int_0^{\infty}\frac{e(\frac{n}{\kappa\lambda}y)e^{it\log\frac{1+y}{y}}}{y^{1-\sigma}(1+y)^{1-\sigma}}dy\right\}\\
&+2(2\pi)^{2\sigma-1}\Im\left\{e^{-2\sigma\pi i}\frac{a(k)\overline{a(l)}}{[k,l]^{2\sigma}}\sigma_{2\sigma-1}(n)e\left(-\frac{\bar{\kappa}n}{\lambda}\right)\int_T^{2T}O(t^{-1-2\sigma})dt\int_0^{\infty}\frac{e(-\frac{n}{\kappa\lambda}y)e^{it\log\frac{1+y}{y}}}{y^{1-\sigma}(1+y)^{1-\sigma}}dy\right\}
\label{eq:SR2}
\end{aligned}\end{equation}
and
\begin{equation}\begin{aligned}
S_2(t;n,k,l)=&2(2\pi)^{2\sigma-1}\Im\left\{e^{-2\sigma\pi i}\frac{a(k)\overline{a(l)}}{[k,l]^{2\sigma}}\Delta_{2\sigma-1}\left(X,\frac{\bar{\kappa}}{\lambda}\right)\int_0^{\infty}\frac{e(\frac{n}{\kappa\lambda}y)e^{it\log\frac{1+y}{y}}}{y^{1-\sigma}(1+y)^{1-\sigma}\log\frac{1+y}{y}}dy\right\}\\
&+2(2\pi)^{2\sigma-1}\Im\left\{e^{-2\sigma\pi i}\frac{a(k)\overline{a(l)}}{[k,l]^{2\sigma}}\overline{\Delta_{2\sigma-1}\left(X,\frac{\bar{\kappa}}{\lambda}\right)}\int_0^{\infty}\frac{e(-\frac{n}{\kappa\lambda}y)e^{it\log\frac{1+y}{y}}}{y^{1-\sigma}(1+y)^{1-\sigma}\log\frac{1+y}{y}}dy\right\},
\label{eq:S2}
\end{aligned}\end{equation}
$S_{2r}$ is the term given by replacing $\sigma_{2\sigma-1}(n)e\left(\frac{\bar{\kappa}n}{\lambda}\right)$ by $\Delta_{2\sigma-1}\left(X,\frac{\bar{\kappa}}{\lambda}\right)$ in $-S_{r1}+S_{r2}$.

Using \eqref{eq:g3} we have
\begin{align*}\begin{aligned}
I_3=&\sum_{k\le M}\sum_{l\le M}\Re\left\{(2\pi)^{2\sigma-1}e^{-2\sigma\pi }\frac{a(k)\overline{a(l)}}{[k,l]^{2\sigma}}\right.\\
&\left.\left(\frac{\zeta(2-2\sigma)\lambda^{2\sigma-2}+\frac{\zeta(2\sigma)}{\lambda^{2\sigma}}X^{2\sigma-1}}{-\pi }\kappa\lambda I_{31}(k,l)+\frac{\kappa\zeta(2\sigma)(2\sigma-1)X^{2\sigma-1}}{\pi \lambda^{2\sigma-1}} I_{32}(k,l)\right)\right\},
\end{aligned}\end{align*}
where
\begin{align*}\begin{aligned}
I_{31}(k,l)&=t^{1-2\sigma}\int_0^{\infty}\frac{\sin\left(2\pi Xy/\kappa\lambda\right)}{y^{2-\sigma}(1+y)^{1-\sigma}\log\frac{1+y}{y}}e^{it\log\frac{1+y}{y}}|^{2T}_{t=T}dy\\
&-\int_T^{2T}t^{-2\sigma}dt\int_0^{\infty}\frac{\sin\left(2\pi Xy/\kappa\lambda\right)}{y^{2-\sigma}(1+y)^{1-\sigma}\log\frac{1+y}{y}}e^{it\log\frac{1+y}{y}}dy\\
&+\int_T^{2T}O(t^{-1-2\sigma})dt\int_0^{\infty}\frac{\sin\left(2\pi Xy/\kappa\lambda\right)}{y^{2-\sigma}(1+y)^{1-\sigma}\log\frac{1+y}{y}}e^{it\log\frac{1+y}{y}}dy\\
&=I_{31M}-E_{31}+E_{31}',
\end{aligned}\end{align*}
say, and
\begin{align*}\begin{aligned}
I_{32}(k,l)&=\int_0^{\infty}\frac{\sin(2\pi Xy/\kappa\lambda)}{y}\int_{\sigma+iT}^{\sigma+2iT}\frac{(u-\sigma)^{1-2\sigma}}{u}\left(\frac{1+y}{y}^u\right)du\ dy\\
&+O\left(\int_0^{\infty}\frac{\sin(2\pi Xy/\kappa\lambda)}{y}\int_{\sigma+iT}^{\sigma+2iT}\frac{(u-\sigma)^{-1-2\sigma}}{u}\left(\frac{1+y}{y}^u\right)du\ dy\right)\\
&=I_{32M}+E_{32},
\end{aligned}\end{align*}
say. As for $I_4$, similar to Atkinson's paper, we first carry out the integration with respect to $t$, then put $xy=\eta$, change the order of differentiation and integration and differentiate with respect to $x$, and again put $\eta/x=y$. Changing the variable temporarily to $\eta$ is necessary to ensure the interchange of differentiation and integration. 

The result is that
\begin{align*}\begin{aligned}
I_4=&\sum_{k\le M}\sum_{l\le M}\Re\left\{C\frac{a(k)\overline{a(l)}}{[k,l]^{2\sigma}}\int_T^{2T}{\left\{\frac{\Gamma(1-\sigma-it)}{\Gamma(\sigma-it)}(g_4(\sigma+it)+\overline{g_4(\sigma-it)})\right\}}dt\right\}\\
&=\sum_{k\le M}\sum_{l\le M}S_4(t;k,l)|^{2T}_{t=T},
\end{aligned}\end{align*}
where
$$S_4(t;k,l)=S_{40}-S_{41}+S_{42}+S_{43}-S_{44}+S_{45},$$
with
$$S_{40}=2\Re\left\{C\frac{a(k)\overline{a(l)}}{[k,l]^{2\sigma}}C_1\int_X^{\infty}\frac{\Delta_{2\sigma-1}\left(x,\frac{\bar{\kappa}}{\lambda}\right)}{x}\int_0^{\infty}\frac{e^{2\pi ixy/\kappa\lambda}}{y^{1-\sigma}(1+y)^{2-\sigma}\log((1+y)/y)}t^{2-2\sigma}e^{it\log\frac{1+y}{y}}dydx\right\},$$
$$S_{41}=2\Im\left\{C_2\frac{a(k)\overline{a(l)}}{[k,l]^{2\sigma}}\int_X^{\infty}\frac{\Delta_{2\sigma-1}\left(x,\frac{\bar{\kappa}}{\lambda}\right)}{x}\int_0^{\infty}\frac{e^{2\pi ixy/\kappa\lambda}e^{it\log\frac{1+y}{y}}\left(1-\sigma+\frac{1}{\log((1+y)/y)}\right)t^{1-2\sigma}}{y^{1-\sigma}(1+y)^{2-\sigma}\log((1+y)/y)}dydx\right\},$$
$$S_{42}=2\Im\left\{C_2\frac{a(k)\overline{a(l)}}{[k,l]^{2\sigma}}\int_X^{\infty}\frac{\Delta_{2\sigma-1}\left(x,\frac{\bar{\kappa}}{\lambda}\right)}{x}\int_0^{\infty}\frac{e^{2\pi ixy/\kappa\lambda}e^{it\log\frac{1+y}{y}}\left((1-2\sigma)(1+y)\right)t^{1-2\sigma}}{y^{1-\sigma}(1+y)^{2-\sigma}\log((1+y)/y)}dydx\right\},$$
and $S_{43}$ (resp.$S_{44}$, $S_{45}$)is similar to $S_{40}$ (resp. $S_{41}$, $S_{42}$) with $\Delta_{2\sigma-1}(x,\bar{\kappa}/\lambda)$ and $e^{2\pi ixy/\kappa\lambda}$ replaced by their complex conjugates.
We decompose
$$S_{40}|^{2T}_{t=T}=J_1(k,l)+J_2(k,l)-J_3(k,l),$$
where
$$J_1(k,l)=2\Re\left\{C\frac{a(k)\overline{a(l)}}{[k,l]^{2\sigma}}C_1\int_{2X}^{\infty}\frac{\Delta_{2\sigma-1}\left(x,\frac{\bar{\kappa}}{\lambda}\right)}{x}\int_0^{\infty}\frac{e^{2\pi ixy/\kappa\lambda}e^{i2T\log\frac{1+y}{y}}}{y^{1-\sigma}(1+y)^{2-\sigma}\log((1+y)/y)}(2T)^{2-2\sigma}dydx\right\},$$
$$J_2(k,l)=2\Re\left\{C\frac{a(k)\overline{a(l)}}{[k,l]^{2\sigma}}C_1\int_X^{2X}\frac{\Delta_{2\sigma-1}\left(x,\frac{\bar{\kappa}}{\lambda}\right)}{x}\int_0^{\infty}\frac{e^{2\pi ixy/\kappa\lambda}e^{2iT\log\frac{1+y}{y}}}{y^{1-\sigma}(1+y)^{2-\sigma}\log((1+y)/y)}(2T)^{2-2\sigma}dydx\right\},$$
and 
\begin{equation}\begin{aligned}J_3(k,l)=2\Re\left\{C\frac{a(k)\overline{a(l)}}{[k,l]^{2\sigma}}C_1\int_X^{\infty}\frac{\Delta_{2\sigma-1}\left(x,\frac{\bar{\kappa}}{\lambda}\right)}{x}\int_0^{\infty}\frac{e^{2\pi ixy/\kappa\lambda}e^{iT\log\frac{1+y}{y}}}{y^{1-\sigma}(1+y)^{2-\sigma}\log((1+y)/y)}T^{2-2\sigma}dydx\right\}.\label{eq:j3}\end{aligned}\end{equation}
Therefore 
\begin{equation*}\begin{aligned}
I_4=\sum_{k\le M}\sum_{l\le M}(J_1(k,l)+J_2(k,l)-J_3(k,l))+\sum_{k\le M}\sum_{l\le M}(-S_{41}+S_{42}+S_{43}-S_{44}+S_{45})|^{2T}_{t=T}.
\end{aligned}\end{equation*}

In order to evaluate the exponential integrals appearing in the above formulas, we use the lemma from Atkinson:

\begin{lemma}(\cite[Lemma 1]{A})
Let $f(z)$, $\phi(z)$ be two functions of the complex variable $z$, and $(a,b)$ a real interval, such that \\
\begin{enumerate}[(i)]
    \item For $a\le x\le b$, $f(x)$ is real and $f''(x)>0$.
    \item For a  certain positive differentiable function $\mu(x)$, defined in $(a,b)$, $f(z)$ and $\phi(z)$ are analytic for $$a\le x\le b,\ |z-x|\le \mu(x)$$
    \item There are positive functions $F(x)$, $\Phi(x)$ defined in $(a,b)$ such that 
$$\phi(z)=O\{\Phi(x)\},\ f'(z)=O\{F(x)\mu^{-1}(x)\},\ \{f''(z)\}^{-1}=O\{(\mu(x))^2(F(x))^{-1}\}$$
for 
$$a\le x\le b,\ |z-x|\le \mu(x)$$
the constants implied in these order results being absolute.
\end{enumerate}

Let $k$ be any real number. and if $f'(x)+k$ has a zero in $(a,b)$ denote it by $x_0$. Let the values of $f(x)$, $\phi(x)$, etc., at $a$, $x_0$, $b$ be characterised by the suffixes $a$, $0$, and $b$ respectively. Then
\begin{align*}
    \int_a^b\phi(x)&\exp 2\pi i\{f(x)+kx\}dx=\phi_0{f''_0}^{-1/2} e^{2\pi i\{f_0+kx_0\}+\frac{1}{2}\pi i}\\
    &+O\left\{\int_a^b\Phi(x)\exp\{-A|k|\mu(x)-AF(x)\}(dx+|d\mu(x)|)\right\}\\
   &+O\left(\Phi_0 \mu_0 {F_0}^{-3/2}\right)+O\left(\frac{\Phi_a}{|f'_a+k|+{f''_a}^{1/2}}\right)+O\left(\frac{\Phi_b}{|f'_b+k|+{f''_b}^{1/2}}\right).
\end{align*}
If $f'(x)+k$ has no zero in $(a,b)$ then the terms involving $x_0$ are to be omitted.
\label{lem-9}
\end{lemma}

\begin{lemma}(\cite[Lemma 2]{A} )\\
Let $\alpha$, $\beta$, $\gamma$, $a$, $b$, $k$, $T$ be real numbers such that $\alpha$, $\beta$, $\gamma$ are positive and bounded, $\alpha\neq1$, $0<a<1/2$, $a<T/(8\pi k)$, $k>0$, $T\ge 1$, and $$b\ge \max\left\{T,\frac{1}{k},-\frac{1}{2}+\sqrt{\frac{1}{4}+\frac{T}{2\pi k}}\right\}.$$
Then,
\begin{equation}\begin{aligned}
\int_a^b\frac{\exp(\pm i\left(T\log((1+y)/y)+2\pi ky)\right)}{y^{\alpha}(1+y)^{\beta}{\log^{\gamma}{((1+y)/y)}}}dy=\frac{T^{1/2}\exp\left(i(TV+2\pi kU-\pi k+\frac{\pi}{4})\right)}{2k\pi^{1/2}V^{\gamma}U^{1/2}(U-\frac{1}{2})^{\alpha}(U+\frac{1}{2})^{\beta}}\\
+O\left(\frac{a^{1-\alpha}}{T}\right)+O\left(\frac{b^{\gamma-\alpha-\beta}}{k}\right)+R(T,k)+O(e^{-CT})\\
+O(e^{-c\sqrt{kT}}(T^{\gamma-\alpha-\beta}+T))+O(k^{-1}e^{-CkT}(b^{\gamma-\alpha-\beta}+T^{\gamma-\alpha-\beta}))
\label{3}
\end{aligned}\end{equation}
uniformly for $|\alpha-1|>\varepsilon$, where $C$ is an absolute constant,
$$U=\sqrt{\frac{1}{4}+\frac{T}{2\pi k}},\ \ \ \ V=2\arcsinh{\sqrt{\frac{\pi k}{2T}}},$$
\begin{equation*}
R(T,k)\ll\left\{
\begin{array}{lr}
T^{(\gamma-\alpha-\beta)/2-1/4}k^{-(\gamma-\alpha-\beta)/2-5/4}\ \ \  &if\ 0<k\le T,  \\
T^{-1/2-\alpha}k^{\alpha-1}\ \ \ \ \ &if\ k\ge T.
\end{array}
\right.
\end{equation*}
The term $O(e^{-CT})$ is to be omitted if $1<\alpha$ or $0<\alpha<1$ with $k\ge T$. When we replace $k$ by $-k$ on the left-hand side of \eqref{3}, then the explicit term and $R(T,k)$ for $k\le T$ on the right-hand side are to be omitted. And if we replace the assumption $k>0$ by $k$ is larger than an absolute constant in $(0,1]$, then the last three error terms on the right-hand side of \eqref{3} are absorbed into the other error terms. When $k\le B$ and $k$ on the left-hand side of \eqref{3} is replaced by $-k$, then the explicit term and the last three error terms on the right-hand side of \eqref{3} are to be omitted.
\label{lem-2}
\end{lemma}

\begin{lemma}
For $\alpha,\ k,\ T\in \mathbf{R}$, then
\begin{equation*}
\begin{aligned}
\int_T^{2T}\frac{\exp -i\left(2x \arcsinh{\sqrt{\frac{\pi k}{2x}}}+2\pi k\sqrt{\frac{1}{4}+\frac{x}{2\pi k}}-\pi k+\frac{\pi}{4}\right)}{x^{\alpha}\cdot\arcsinh{\sqrt{\frac{\pi k}{2x}}}\cdot\left(\frac{1}{4}+\frac{x}{2\pi k}\right)^{1/4}}dx=O(T^{-\alpha+3/4}).
\end{aligned}\end{equation*}

\label{lem-3}
\end{lemma}

\begin{proof}(proof of Lemma \ref{lem-3} )
    We apply Lemma \ref{lem-9} with $a=T$, $b=2T$ and $$\phi(x)=x^{-\alpha}\left({\arcsinh\sqrt{\frac{\pi k}{2x}}}\right)^{-1}\left(\frac{1}{4}+\frac{x}{2\pi k}\right)^{-1/4},$$ $$f(x)=-2x \arcsinh{\sqrt{\frac{\pi k}{2x}}}-2\pi k\sqrt{\frac{1}{4}+\frac{x}{2\pi k}}+\pi k-\frac{\pi}{4}.$$

    We have then 
    $$f'(x)=-2\arcsinh\sqrt{\frac{\pi k}{2x}}+\frac{\sqrt{\pi k}}{\sqrt{{\pi k}+2x}}-\frac{1}{2}\left(\frac{1}{4}+\frac{x}{2\pi k}\right)^{-1/2},$$
    $$f''(x)=\frac{\sqrt{\pi k}(x+\pi k)}{\sqrt{2x+\pi k}(2x^2+\pi k x)}+\frac{1}{8\pi k}\left(\frac{1}{4}+\frac{x}{2\pi k}\right)^{-3/2},$$
    
   since $f'(T)=O(T^{-1/2})$ and $f''(T)=O(T^{-3/2})$, we have $$O\left\{\frac{1}{(|f'_T|+{f''_T}^{1/2})}\right\}=O\left(\frac{1}{T^{-1/2}+T^{-3/4}}\right)=O(T^{1/2})$$
   and $$\arcsinh\left(\sqrt{\frac{\pi k}{2x}}\right)=O(x^{-1/2}).$$
    
    So that we may take $$\Phi(x)=x^{-\alpha+1/4},\ \mu(x)=x/2,\ F(x)=\sqrt{x},$$
    then $O\left(\frac{\Phi_a}{|f'_a+k|+{f''_a}^{1/2}}\right)=O(T^{-\alpha+1/4}\cdot T^{1/2}).$

    Since $f''(x)>0$ for $x>0$ and $\lim_{x\to\infty}f'(x)\to 0^-$, so $f'(x)$ has no zero inside the area $(T,2T)$, and the main term and second error term can be omitted. Moreover, the first error term is also to be omitted.

\end{proof}

\section{Evaluation of $I_1$, $I_2$ and $I_3$}
We apply Lemma \ref{lem-2} to  $I_1$, $I_2$ and $I_3$ in this section.

For $I_1$, firstly, we can see that $S_1(t;n,k,l)|^{2T}_{t=T}$ is absolutely convergent, by
\begin{equation}\begin{aligned}\label{1}
&\int_0^{\infty}\frac{e^{i2\pi ny/q}[\exp(i2T\log{((1+y)/y)})-\exp(iT\log{((1+y)/y)})]}{y^{1-\sigma}(1+y)^{1-\sigma}\log{((1+y)/y)}}dy\\
&=\int_0^{\infty}\frac{e^{i2\pi ny/q}e^{\frac{3}{2}Ti\log{((1+y)/y)}}\left[e^{i\frac{T}{2}\log{((1+y)/y)}}-e^{-i\frac{T}{2}\log{((1+y)/y)}}\right]}{y^{1-\sigma}(1+y)^{1-\sigma}\log{((1+y)/y)}}dy\\
&=\int_0^{\infty}\frac{e^{i2\pi ny/q}e^{\frac{3}{2}Ti\log{((1+y)/y)}}\left[2i\sin\left(\frac{T}{2}\log\frac{1+y}{y}\right)\right]}{y^{1-\sigma}(1+y)^{1-\sigma}\log{((1+y)/y)}}dy\\
&\ll O\left(\int_0^{\infty}\frac{e^{i2\pi ny/q}e^{\frac{3}{2}Ti\log{((1+y)/y)}}\left[2\left(\frac{T}{2}\log\frac{1+y}{y}\right)\right]}{y^{1-\sigma}(1+y)^{1-\sigma}\log{((1+y)/y)}}dy\right)<_T\infty\ (by\ Lemma \ref{lem-2}).
\end{aligned}\end{equation}
Hence, for a fixed large $b$, applying Lemma \ref{lem-2}, 
\begin{equation}\begin{aligned}\label{2}
&\int_0^{b}\frac{e^{i2\pi ny/\kappa\lambda}\exp(it\log{((1+y)/y)})}{y^{1-\sigma}(1+y)^{1-\sigma}\log{((1+y)/y)}}dy
=\frac{T^{1/2}\exp\left(i(TV+2\pi \frac{n}{\kappa\lambda}U-\pi n/\kappa\lambda+\frac{\pi}{4})\right)}{2\frac{n}{\kappa\lambda}\pi^{1/2}VU^{1/2}(U^2-{\frac{1}{2}}^2)^{1-\sigma}}\\
&+O\left(\frac{b^{2\sigma-1}}{\frac{n}{\kappa\lambda}}\right)+O\left(T^{\sigma-3/4}{\left(\frac{n}{\kappa\lambda}\right)}^{-\sigma-3/4}\right).
\end{aligned}\end{equation}
So, by \eqref{eq:SumI1} we can write $I_1$ with
\begin{align*}\begin{aligned}
I_1+S_{r1}-S_{r2}&=\sum_{k\le M}\sum_{l\le M}\sum_{n\le 2X}S_1(2T;n,k,l)\\
&-\sum_{k\le M}\sum_{l\le M}\sum_{X< n\le 2X}S_1(2T;n,k,l)-\sum_{k\le M}\sum_{l\le M}\sum_{n\le X}S_1(T;n,k,l)\\
&=I_{11}-I_{12}-I_{13}
\end{aligned}\end{align*}
above \eqref{1} and \eqref{2}, for $A\kappa\lambda T<X<A'\kappa\lambda T$, we obtain
\begin{equation}\begin{aligned}
I_{13}=\Sigma_1(T,Y)+O\left(T^{1/4-\sigma}C\sum_{k\le M}\sum_{l\le M}\frac{|a(k)\overline{a(l)}|}{[k,l]^{2\sigma}}(\kappa\lambda)^{\sigma+3/4}\right),
\label{59}
\end{aligned}\end{equation}
where \begin{align*}\begin{aligned}
&\Sigma_1(T,Y)=\sum_{k,l\le M}\sum_{n\le \kappa\lambda Y}\Im\left\{C\frac{a(k)\overline{a(l)}}{[k,l]^{2\sigma}}(\kappa\lambda)^{\sigma}\frac{\sigma_{2\sigma-1}(n)}{2^{\sigma+1}n^{\sigma}\pi^{\sigma-1/2}}e^{2\pi in\bar{\kappa}/\lambda}T^{1/2-\sigma}\right.\\
&\times\left.\left(\arcsinh{\sqrt{\frac{\pi n}{2T\kappa\lambda}}}\right)^{-1}\left(\frac{1}{4}+\frac{T\kappa\lambda}{2\pi n}\right)^{-1/4}\exp\left(i\left(f\left(T,\frac{n}{\kappa\lambda}\right)-\pi n/\kappa\lambda+\pi/2\right)\right) \right\},
\end{aligned}\end{align*}
here$$f(T,u)=2T\arcsinh{\sqrt{\frac{\pi u}{2T}}}+\sqrt{2\pi u T+\pi^2u^2}-\frac{\pi}{4}$$
Replacing $\Sigma_1(T,Y)$ to $\Sigma_1(2T,2Y)$ we have the same type formula holds for $I_{11}$.

For $S_{r2}$, by Lemma \ref{lem-2}, note that the order of the integral for $y$ in the double integral in \eqref{eq:SR2} is $O(T^{-3/4+\sigma})$. So the double integral in $S_{r2}$ is $O(T^{-3/4-\sigma})$, it is absorbed by the error term of \eqref{59} for $\sigma>1/4$.

On the double integral in $I_{r1}$, by using Lemma \ref{lem-2} we have
$$I_{r1}=C(\kappa,\lambda,n,\pi)\overline{\int_T^{2T} x^{-2\sigma}x^{1/2}x^{\sigma-1}\frac{\exp -i\left(2x \arcsinh{\sqrt{\frac{\pi m}{2x}}}+2\pi m\sqrt{\frac{1}{4}+\frac{x}{2\pi m}}-\pi m+\frac{\pi}{4}\right)}{\arcsinh{\sqrt{\frac{\pi m}{2x}}}\times\left(\frac{1}{4}+\frac{x}{2\pi m}\right)^{1/4}}dx},$$
where $C(\kappa,\lambda,n,\pi)$ is a constant and $m=\frac{n}{\kappa\lambda}$.

By Lemma \ref{lem-3}, the double integral in $I_{r1}$ is $O_{\kappa,\lambda}(T^{1/4-\sigma})$.

Next, applying \eqref{O:d} and Lemma\ref{lem-2} to \eqref{eq:S2}, we obtain
\begin{align*}
I_2\ll\sum_{k\le M}\sum_{l\le M}\frac{|Ca(k)\overline{a(l)}|}{[k,l]^{2\sigma}}\kappa^{2\sigma-1+\frac{1}{5-2\sigma}+\varepsilon}\lambda^{2\sigma-1+\frac{4-2\sigma}{5-2\sigma}+\varepsilon}T^{2\sigma-3/2+1/(5-4\sigma)+\varepsilon} .
\end{align*}
For $I_{31M}$, we can divide the integral into two parts: $(0,y)$ and $[y,\infty),\ y=\frac{1}{2x}$, which we denote by $I_{31M}^{'}$ and $I_{31M}^{''}$.
We have
$$I_{31M}^{'}=\int_0^y \frac{\sin(T\log(\frac{y+1}{y}))}{y(1+y)}\frac{y^{\sigma}(1+y)^{\sigma}}{\log(\frac{y+1}{y})}\frac{\sin(2\pi Xy/\kappa\lambda)}{y}dy.$$
Using the second mean-value Theorem twice and $\eta\sim x^{-1},\ x\sim T$, we obtain $$I_{31M}^{'}\ll O(T^{-\sigma}).$$
 By Atkinson's lemma, we have$$I_{31M}^{''}\ll O(T^{-1/2}).$$
Combining the estimates for $I_{31M}'$ and $I_{31M}''$, we can obtain an estimate for $E_{31}$. And $E_{31}'$ can be estimated by the same way as the estimate for $S_{r2}$. They are all absorbed into $I_{31M}$.

About $I_{32M}$, we have
$$I_{32M}=\int_0^{\infty}\frac{\sin(\frac{2\pi Xy}{\kappa\lambda})}{y^{2-2\sigma}}\int_{\sigma+iT}^{\sigma+2iT}u^{-1}\left(\frac{1+y}{y}\right)^{u}t^{1-2\sigma}dtdy.$$
In Atkinson[1] the integration with respect to $u$ is from $1/2-iT$ to $1/2+iT$, and the main term appears from the residue at $u = 0$ when the path of integration is deformed. In the present situation, however, the integration with respect to $u$ is from $\sigma + iT$ to $\sigma+ 2iT$, hence the residue does not appear. We obtain
$$I_{32M}\ll O(T^{-\sigma}).$$
Similarly we can obtain $E_{32}$ is absorbed by $I_{32M}$.

We obtain
$$I_3\ll \sum_{k\le M}\sum_{l\le M}\frac{|a(k)\bar{a(l)}|}{[k,l]^{2\sigma}}\kappa\lambda^{2\sigma-1}T^{-\sigma}\ll O(M^{1-2\sigma+\varepsilon}T^{-\sigma}).$$
Next, we find that the error term on the right-hand side of \eqref{59} is $O(M^{1-2\sigma+\varepsilon}T^{1/4-\sigma})$. Similarly, we have
$$I_2\ll O(M^{1-2\sigma+\varepsilon}T^{2\sigma-3/2+1/(5-4\sigma)+\varepsilon}).$$
Collecting all the results obtained in this section, we have
\begin{align*}
I_1-I_2+I_3=\Sigma_1(2T,2Y)-\Sigma_1(T,Y)-I_{12}+O\left({M^{1-2\sigma+\varepsilon}\left(T^{1/4-\sigma}+T^{2\sigma-3/2+1/(5-4\sigma)+\varepsilon}\right)}\right).
\end{align*}

\section{Evaluation of $I_4$}
Now the only task remaining to us is to evaluate $I_4$. We apply the last several lines of the statement of Lemma \ref{lem-2} to the inner integral of the right-hand side of $S_{41}$ and $S_{42}$:
\begin{align*}
-S_{41}+S_{42}&\ll \frac{|a(k)\bar{a(l)}|}{[k,l]^{2\sigma}}t^{1-2\sigma}\int_x^{\infty}\frac{|\Delta_{2\sigma-1}(x,\frac{\bar{\kappa}}{\lambda})|}{x}\left(\frac{\kappa\lambda}{x}\right)^{1/2}dx\\
&=\frac{|a(k)\bar{a(l)}|}{[k,l]^{2\sigma}}(\lambda\kappa)^{1/2}t^{1-2\sigma}\sum_{j=1}^{\infty}\int_{2^{j-1}x}^{2^jx}x^{-3/2}|\Delta_{2\sigma-1}(x,\frac{\bar{\kappa}}{\lambda})|dx.
\end{align*}
By the Cauchy-Schwarz inequality, each integral on the right side is
$$\ll \left(\int_{2^{j-1}x}^{2^jx}x^{-3}dx\right)^{1/2}\left(\int_{2^{j-1}x}^{2^jx}\left|\Delta_{2\sigma-1}(x,\frac{\bar{\kappa}}{\lambda})\right|^2dx\right)^{1/2}.$$
\cite{K} gives for $1\le k \le u$, $-1/2<a\le 0$,
\begin{equation*}\begin{aligned}
\int_{u/2}^{u}\left|\Delta_{2\sigma-1}(x,\frac{\bar{\kappa}}{\lambda})\right|^2dx\ll &\lambda\left(u^{1/2+2\sigma}-\left(\frac{u}{2}\right)^{1/2+2\sigma}\right)\\
&+O(\lambda^2u^{1+\varepsilon})+O(\lambda^{3/2}u^{3/4+\sigma+\varepsilon}).
\end{aligned}\end{equation*}
And since $T\sim \frac{X}{\kappa\lambda}$ we obtain
\begin{align*}\begin{aligned}
-S_{41}+S_{42}&\ll \frac{|a(k)\bar{a(l)}|}{[k,l]^{2\sigma}}t^{1-2\sigma}(\lambda\kappa)^{1/2}\sum_{j=1}^{\infty}(2^jX)^{-1}\\
&\times\left\{\lambda^{1/2}(2^jX)^{1/4+\sigma}+\lambda(2^jX)^{1/2+\varepsilon}+\lambda^{3/4}(2^jX)^{3/8+\sigma/2+\varepsilon}\right\}\\
&\ll \frac{|a(k)\bar{a(l)}|}{[k,l]^{2\sigma}}\left(\kappa^{5/4-\sigma}\lambda^{7/4-\sigma}T^{1/4-\sigma}+\kappa^{\varepsilon}\lambda^{1/2+\varepsilon}T^{1/2-2\sigma\varepsilon}+\right.\\
&\left.\kappa^{9/8-\sigma/2}\lambda^{15/8-\sigma/2+\varepsilon}T^{3/8-3\sigma/2+\varepsilon}\right),
\end{aligned}\end{align*}
and the same estimate also holds for $S_{43}-S_{44}+S_{45}$.
\section{Application of Atkinson's Third Lemma}
To evaluate the term $I_4$ further, we must substitute Roy's summation formula (Lemma \ref{lem-1}) to appear in the $J_3$.
By using the usual asymptotic expansions of Bessel functions, we can rewrite \eqref{eq:d} to the following:
\begin{equation}\begin{aligned}
&\Delta_{2\sigma-1}\left(x;\frac{\kappa}{\lambda}\right)=\frac{\lambda^{1/2}}{\sqrt{2}\pi}x^{\sigma-1/4}\sum_{n>0}\frac{\sigma_{2\sigma-1}(n)}{n^{1/4+\sigma}}e\left(-\frac{\kappa n}{\lambda}\right)\\
&\times\left(\cos\left(\frac{4\pi\sqrt{nx}}{\lambda}-\frac{\pi}{4}\right)-\frac{((16\sigma^2-1)-1)\lambda}{32\pi\sqrt{nx}}\sin\left(\frac{4\pi\sqrt{nx}}{\lambda}-\frac{\pi}{4}\right)\right)+O(x^{\sigma-5/4}).
\label{eq:d_bessel}
\end{aligned}\end{equation}
The infinite series in Voronoi's formula converges boundedly in (any fixed) finite intervals only, so Atkinson's term-wise integration process is not valid as it is. Of course, our case meets the same situation. Therefore, we can do something similar to Matsumoto's \cite{M} ``finite cut" form of $J_3$:
\begin{align*}\begin{aligned}
&\lim_{b\to \infty}\int_X^{b^2}\Delta_{2\sigma-1}\left(x,\frac{\bar{\kappa}}{\lambda}\right)T^{3/2-\sigma}x^{-\sigma-1}\arcsinh^{-1}\left(\sqrt{\frac{\pi x}{2\kappa\lambda T}}\right)\\
&\times\left(\frac{T\kappa\lambda}{2\pi x}+1/4\right)^{-1/4}\left(\frac{\kappa\lambda}{2\pi}\right)^{\sigma-1}\left(\sqrt{\frac{\pi x}{2\kappa\lambda T}}\right)\left(\sqrt{\frac{T\kappa\lambda}{2\pi x}+1/4}+1/2\right)^{-1}\\
&\times\exp{(i\left(f(T,x)-\pi x+\pi/2\right)}+O\left(T^{\frac{4\sigma-3}{2(5+4\sigma)}+1-2\sigma}\right),
\end{aligned}\end{align*}
let$$\phi_a(T,x)=x^{-a}\arcsinh^{-1}\left(x\sqrt{\frac{\pi}{2T}}\right)\left(\sqrt{\frac{T}{2\pi x^2}+1/4}+1/2\right)^{-1}\left(\frac{T}{2\pi x^2}+1/4\right)^{-1/4}.$$

In view of Lebesgue's bounded convergence theorem, we can do the term-wise integration. Hence, changing the variable $x$ to $x^2$, and substituting the formula \eqref{eq:d_bessel} into the right side of \eqref{eq:j3}. We obtain
\begin{align*}
J_3(k,l)=&J_{31}(k,l)-J_{32}(k,l)+J_{33}\\
&+O\left(\frac{|a(k)\overline{a(l)}|}{[k,l]^{2\sigma}}\left(\kappa^{5/4-\sigma}\lambda^{7/4-\sigma}T^{1/4-\sigma}+\kappa^{\varepsilon}\lambda^{1/2+\varepsilon}T^{1/2-2\sigma+\varepsilon}\right.\right.\\
&+\left.\left.\kappa^{9/8-\sigma/2}\lambda^{15/8-\sigma/2+\varepsilon}T^{3/8-3\sigma/2+\varepsilon}\right)\right),
\end{align*}
where 
\begin{align*}
J_{31}(k,l)=&\Re\left\{\frac{|a(k)\overline{a(l)}|}{[k,l]^{2\sigma}}\frac{T^{\frac{3}{2}-\sigma}}{2^{\sigma-1/2}\pi^{\sigma+1/2}}\lambda^{\frac{1}{4}+\sigma}\kappa^{\sigma-\frac{1}{4}}\right.\\
&\times \lim_{b\to \infty}\int_{\sqrt{X/\kappa\lambda}}^{b}\sum_{n=1}^{\infty }\frac{\sigma_{2\sigma-1}}{n^{1/4+\sigma}}\cos\left(4\pi x\sqrt{\frac{\kappa n}{\lambda}}-\pi/4\right)\left.\phi_{3/2}(T,x^2)\exp{(i\left(f(T,x^2)-\pi x^2+\frac{\pi}{2}\right)}dx\right\},
\end{align*}
\begin{align*}
J_{32}(k,l)=&\Re\left\{\frac{|a(k)\overline{a(l)}|}{[k,l]^{2\sigma}}\frac{T^{\frac{3}{2}-\sigma}}{2^{\sigma-1/2}\pi^{\sigma+1/2}}\lambda^{\frac{3}{4}+\sigma}\kappa^{\sigma-\frac{3}{4}}\frac{(16\sigma^2-1)}{32\pi n^{1/2}}\right.\\
&\times\lim_{b\to \infty}\int_{\sqrt{X/\kappa\lambda}}^{b}\sum_{n=1}^{\infty }\frac{\sigma_{2\sigma-1}}{n^{1/4+\sigma}}\sin\left(4\pi x\sqrt{\frac{\kappa n}{\lambda}}-\pi/4\right)\left.\phi_{5/2}(T,x^2)\exp{(i\left(f(T,x^2)-\pi x^2+\frac{\pi}{2}\right)}dx\right\},
\end{align*}
\begin{align*}
J_{33}(k,l)=O\left(\frac{|a(k)\overline{a(l)}|}{[k,l]^{2\sigma}}T^{3/2-\sigma}\int_X^{\infty}x^{\sigma-9/4}|\phi_{2\sigma}(T,\sqrt{x/\kappa\lambda})|dx\right)\\
\ll O\left(\frac{|a(k)\overline{a(l)}|}{[k,l]^{2\sigma}}T^{3/2-\sigma}\int_X^{\infty}x^{\sigma-9/4}x^{-\sigma}dx\right)\ll O(T^{1/4-\sigma}).
\end{align*}
For $J_{31}$, we consider a finite truncation of the integral and apply Atkinson's third lemma.

\begin{lemma}(\cite[Lemma 3]{A})
Let A, B be any two fixed constants, $0<A<B$, and we assume $AT^{1/2}<a<BT^{1/2}$. Then,
\begin{equation*}\begin{aligned}
\int_a^b& \phi_{\alpha}(T,x)exp\left\{i\left(\pm4\pi x \sqrt{n}-2T \cdot \arcsinh(x\sqrt{\pi/2T})-\sqrt{2\pi x^2T+\pi^2x^4}+\pi x^2\right)\right\}dx.\\
&=4\pi \delta T^{-1}n^{(\alpha-1)/2}\log^{-1}(T/2\pi n)((T/2\pi)-n)^{(3/2)-\alpha}\exp(i(T-T\cdot \log(T/3\pi n)-2\pi n+\pi/4))\\
&+O(\delta n^{(\alpha-1)/2}((T/2\pi)-n)^{1-\alpha}T^{-3/2})\\
&+O(T^{-\alpha/2}\min(1,|a-(a^2+2T/\pi)^{1/2}\pm2\sqrt{n}|^{-1}))\\
&+O(b^{-a}(n^{1/2}+O(T/b))^{-1})+O(e^{-CT-C\sqrt{nT}}),
\end{aligned}\end{equation*}
with a large positive constant $C$, where $\delta=1$ if $n\le T/2\pi,\  na^2\le ((T/2\pi)-n)^2\le nb^2$ and the double sign takes the sign $+$, and in other case $\delta=0$.

\label{lem-4}
\end{lemma}
Now we substitute the result of Lemma \ref{lem-4} to the right-hand side of $J_{31}$, and obtain
$$J_{31}=M_1(k,l)+M_2(k,l)+M_3(k,l)+M_4{k,l}+M_5(k,l),$$
where
$$M_1(k,l)=\frac{T^{3/2-\sigma}}{\pi^{1/2+\sigma}2^{\sigma-1/2}}\Re\left\{\frac{a(k)\overline{a(l)}}{[k,l]^{2\sigma}}\kappa^{\sigma-1/4}\lambda^{\sigma+1/4}\sum_{n\le Z}\frac{\sigma_{2\sigma-1}(n)e\left(\frac{-\kappa n}{\lambda}\right)}{n^{\sigma+1/4}}\right.$$
$$\left.\times\left(\frac{\kappa n}{\lambda}\right)^{1/4}\frac{4\pi \exp(i\cdot g(T,\kappa n/\lambda))}{T\log(\lambda T/2\pi\kappa n)}\right\},$$
with $Z=(\lambda/\kappa)\xi(T,X/\kappa\lambda)=(\lambda/\kappa)\xi(T,Y)$,
\begin{equation}\begin{aligned}
&M_2(k,l)\ll T^{-\sigma}\frac{|a(k)\overline{a(l)}|}{[k,l]^{2\sigma}}\kappa^{\sigma-1/4}\lambda^{\sigma+1/4}\sum_{n\le Z}\frac{\sigma_{2\sigma-1}(n)e\left(\frac{-\kappa n}{\lambda}\right)}{n^{\sigma+1/4}}\left(\frac{\kappa n}{\lambda}\right)^{1/4}(T/2\pi-n)^{-1/2},\\
&M_3(k,l)\ll T^{3/4-\sigma}\frac{|a(k)\overline{a(l)}|}{[k,l]^{2\sigma}}\kappa^{\sigma-1/4}\lambda^{\sigma+1/4}\sum_{n=1}^{\infty}\frac{\sigma_{2\sigma-1}(n)e\left(\frac{-\kappa n}{\lambda}\right)}{n^{\sigma+1/4}}\min\left(1,\frac{1}{|2\sqrt{\kappa n/\lambda}-2\sqrt{\xi(T,Y)}|}\right),\label{w3}
\end{aligned}\end{equation}
and $$M_4\ll T^{3/2-\sigma}e^{-CT}\frac{|a(k)\overline{a(l)}|}{[k,l]^{2\sigma}}\kappa^{\sigma-1/4}\lambda^{\sigma+1/4}\sum_{n=1}^{\infty}\frac{\sigma_{2\sigma-1}(n)e\left(\frac{-\kappa n}{\lambda}\right)}{n^{\sigma+1/4}}e^{-C\sqrt{nT}}, $$
$$M_5(k,l)\ll b^{-3/2}\sum_{n=1}^{\infty}\frac{\sigma_{2\sigma-1}(n)e\left(\frac{-\kappa n}{\lambda}\right)}{n^{\sigma+1/4}}\frac{1}{n^{1/2}+O(T/b)}.$$
Here $M_5$ is convergent, and it tends to 0 when $b$ tends to $\infty$.
In $J_{32}$, it is not necessary to consider the finite truncation of the integral, because the infinite series is absolutely convergent, so we can write $J_{32}$ as 
$$M_6(k,l)+M_7(k,l)+M_8(k,l)+M_9(k,l),$$
where
\begin{equation*}\begin{aligned}M_6(k,l)=\frac{T^{3/2-\sigma}(16\sigma^2-1)}{\pi^{3/2+\sigma}2^{\sigma+9/2}}\Re\left\{\frac{a(k)\overline{a(l)}}{[k,l]^{2\sigma}}\kappa^{\sigma-3/4}\lambda^{\sigma+3/4}\sum_{n\le Z}\frac{\sigma_{2\sigma-1}(n)e\left(\frac{-\kappa n}{\lambda}\right)}{n^{\sigma+1/4}}\right.\\
\left.\times\left(\frac{\kappa n}{\lambda}\right)^{3/4}\frac{4\pi \exp(ig(T,\kappa n/\lambda))}{T\log(\lambda T/2\pi\kappa n)(T/2\pi-\kappa n/\lambda)}\right\}
\end{aligned}\end{equation*}
and $M_7(k,l),\ M_8(k,l),\ M_9(k,l)$ are similar to $M_2$, $M_3$, $M_4$ just replacing the factor $\kappa^{\sigma-1/4}\lambda^{\sigma+1/4}$ by $\kappa^{\sigma-3/4}\lambda^{\sigma+3/4}$, and the factor $n^{\sigma+1/4}$ by $n^{\sigma+3/4}$. In $M_7$ replace the factor $(T/2\pi-\kappa n/\lambda)^{1/2}$ by $(T/2\pi-\kappa n/\lambda)^{3/2}$ and in $M_8$ replace the factor $T^{3/4-\sigma}$ by $T^{1/4-\sigma}$.
We can obtain $$M_2,M_4,M_6,,M_7,M_9\ll \frac{|a(k)\overline{a(l)}|}{[k,l]^{2\sigma}}\lambda^{1/2+\sigma}T^{-\sigma}\log(T),$$
$$M_8\ll\frac{|a(k)\overline{a(l)}|}{[k,l]^{2\sigma}}T^{1/4-\sigma}.$$
For $M_3$, we can show that
$$M_3\ll \frac{a(k)\overline{a(l)}}{[k,l]^{2\sigma}} T^{1-2\sigma}\log(T)$$
holds. To prove this, we cut the right side of \eqref{w3} by 
$$\sum_{n\le Z/2}+\sum_{Z/2<n\le Z-\sqrt{Z}}+\sum_{Z-\sqrt{Z}<n\le Z+\sqrt{Z}}+\sum_{Z+\sqrt{Z}<n\le 2Z}+\sum_{n> 2Z}$$
and evaluate each sum in a way similar to that of Atkinson\cite{A}, the details being omitted.

Collecting the above all, we have 
\begin{equation*}\begin{aligned}
J_3(k,l)=M_1(k,l)+O_M\left(\frac{|a(k)\overline{a(l)}|}{[k,l]^{2\sigma}}\left\{T^{1-2\sigma}\log(T)+T^{1/4-\sigma}\right\}\right).
\end{aligned}\end{equation*}
We can show a similar result for $J_1(k,l)$.
Since $$\sum_{k\le M}\sum_{l\le M}M_1(k,l)=-\Sigma_2(T,\xi(T,Y)),$$
we obtain
\begin{equation*}\begin{aligned}
I_4=&-\Sigma_2(2T,\xi(2T,2Y))+\Sigma_2(T,\xi(T,Y)\\
&+\sum_{k\le M}\sum_{l\le M}J_2(k,l)+O_M(T^{1-2\sigma}\log(T))+O_M(T^{1/4-\sigma}).
\end{aligned}\end{equation*}

\section{On $I_{12}$ and $J_2(k,l)$}
Summing up all the results we got, we obtain
\begin{equation*}\begin{aligned}
\int_T^{2T}&\left|\zeta(\sigma+it)A(\sigma+it)\right|^2dt=\mathcal{M}(2T,A)+\Sigma_1(2T,2Y)+\Sigma_2(2T,\xi(2T,2Y))\\
&-\mathcal{M}(T,A)-\Sigma_1(T,Y)-\Sigma_2(T,\xi(T,Y))\\
&-I_{12}-\sum_{k\le M}\sum_{l\le M}J_2(k,l)+O_M(T^{1-2\sigma}\log(T))+O_M(T^{1/4-\sigma}).
\end{aligned}\end{equation*}
Hence our task is to consider $I_{12}$ and $\sum\sum J_2(k,l)$. 
Next, we use the method in \cite{I3} to verify the cancellation process.

Let
$$S_{13}=2(2\pi)^{2\sigma-1}\Im\left\{e^{-2\sigma\pi i}\frac{a(k)\overline{a(l)}}{[k,l]^{2\sigma}}\sigma_{2\sigma-1}(n)e\left(\frac{\bar{\kappa}n}{\lambda}\right)t^{1-2\sigma}\int_0^{\infty}\frac{e(\frac{n}{\kappa\lambda}y)e^{-it\log\frac{1+y}{y}}}{y^{1-\sigma}(1+y)^{1-\sigma}\log\frac{1+y}{y}}dy\right\}$$
and we can rewrite $S_1(2T;n,k,l)$ as 
$$S_{2T;n,k,l}=S_{11}(2T;n,k,l)-S_{13}(2T;n,k,l)+S_{13}(2T;n,k,l)+S_{12}(2T;n,k,l).$$
We denote the last two terms on the right-hand side of the above equation by $R_1$. For the first two terms, we have
\begin{align*}
    &\int_0^{\infty}\frac{e(\frac{n}{\kappa\lambda}y)}{y^{1-\sigma}(1+y)^{1-\sigma}\log\frac{1+y}{y}}(e^{i2T\log\frac{1+y}{y}}-e^{-i2T\log\frac{1+y}{y}})dy\\
    &=\int_{0}^{\infty}\frac{e(\frac{n}{\kappa\lambda}y)}{y^{1-2\sigma}(1+y)}\int_{\sigma-2iT}^{\sigma+2iT}\left(\frac{1+y}{y}\right)^{u}dudy\\
    &=\int_{\sigma-2iT}^{\sigma+2iT}h\left(u,\frac{n}{\kappa\lambda}\right),
\end{align*}
since the $h(u,n/\kappa\lambda)$ is uniformly convergent in $\Re(u)<2\sigma$. So,
\begin{align*}
    I_{12}&=2(2\pi)^{2\sigma-1}\sum_{k\le M}\sum_{l\le M}\Im\left\{e^{-2\pi i \sigma}(2T)^{1-2\sigma}\int_{\sigma-2iT}^{\sigma+2iT}\frac{a(k)\overline{a(l)}}{[k,l]^{2\sigma}}\sum_{X<n<2X}\sigma_{2\sigma-1}e^{2\pi i \Bar{\kappa} n/\lambda}h\left(u,\frac{n}{\kappa\lambda}\right)du\right\}+R_1.
\end{align*}
Then we can express the inner sum as what we do in \eqref{eq:deh}, and we apply \eqref{eq:D}, obtain
\begin{align*}
    \sum_{X<n<2X}\sigma_{2\sigma-1}e^{2\pi i \Bar{\kappa} n/\lambda}h\left(u,\frac{n}{\kappa\lambda}\right)&=\int_X^{2X}h\left(u,\frac{n}{\kappa\lambda}\right)\left(\frac{\zeta(2-2\sigma)}{\lambda^{2-2\sigma}}+\frac{\zeta(2\sigma)}{\lambda^{2\sigma}}X^{2\sigma-1}\right)dx\\
    &+h\left(u,\frac{n}{\kappa\lambda}\right)\Delta_{2\sigma-1}\left.\left(x,\frac{\Bar{\kappa}}{\lambda}\right)\right|_x^{2X}-\int_X^{2X}\frac{\partial{h\left(u,\frac{x}{\kappa\lambda}\right)}}{\partial{x}}\Delta_{2\sigma-1}\left(x,\frac{\bar{\kappa}}{\lambda}\right)dx.
\end{align*}
Denoting the first term of the right side by $R_2$ and the second term of the right side by $R_3$, we can rewrite $I_{12}$ as
\begin{equation*}\begin{aligned}
I_{12}&=-2(2\pi)^{2\sigma-1}\sum_{k\le M}\sum_{l\le M}\Im\left\{e^{-2\pi i \sigma}(2T)^{1-2\sigma}\int_{\sigma-2iT}^{\sigma+2iT}\frac{a(k)\overline{a(l)}}{[k,l]^{2\sigma}}\int_X^{2X}\frac{\partial{h\left(u,\frac{x}{\kappa\lambda}\right)}}{\partial{x}}\Delta_{2\sigma-1}\left(x,\frac{\bar{\kappa}}{\lambda}\right)dxdu\right\}\\
&+R_1+R_2+R_3.
\end{aligned}\end{equation*}
Changing the integrations on the right side and using
\begin{align*}
    \int_{\sigma-i2T}^{\sigma+2iT}&\frac{\partial{h\left(u,\frac{x}{\kappa\lambda}\right)}}{\partial{x}}du=\frac{1}{x}\int_{0}^{\infty}\frac{e^{2\pi ixy/\kappa\lambda}e^{it\log((1+y)/y)}}{y^{1-\sigma}(1+y)^{2-\sigma}\log((1+y)/y)}\\
    &\times\left.\left(it-\sigma+1-\frac{1}{\log((1+y)/y)}+(1-2\sigma)(1+y)\right)\right|_{t=-2T}^{2T}dy.
\end{align*}
Here, let's focus on the
\begin{align*}
    e^{it\log((1+y)/y)}&\left.\left(it-\sigma+1-\frac{1}{\log((1+y)/y)}+(1-2\sigma)(1+y)\right)\right|_{t=-2T}^{2T}\\
   & =2iTe^{2iT\log((1+y)/y)}+2iTe^{-2iT\log((1+y)/y)}\\
    &+e^{it\log((1+y)/y)}\left.\left(-\sigma+1-\frac{1}{\log((1+y)/y)}+(1-2\sigma)(1+y)\right)\right|_{t=-2T}^{2T}
\end{align*}
and denote the contribution of the second term and the third term on the right side by $R_4$ and $R_5$. We obtain
\begin{equation*}
    I_{12}=-\sum_{k\le M}\sum_{l\le M}J_2(k,l)+\sum_{i=1}^5R_i.
\end{equation*}
 Next, we estimate every $R_i$. Basically, the estimates of $R_1$, $R_3$, $R_4$ and $R_5$ are similar to the valuation we have discussed before, that is $R_3$ can be treated similarly to $I_2$, and $R_4$ similarly to $S_{43}$, $R_5$ similarly to $-S_{41}+S_{42}$ and $-S_{44}+S_{45}$. We obtain
 $$R_1+ R_3+R_4+R_5\ll_{M}O(T^{2\sigma-3/2+1/(5-4\sigma)+\varepsilon}+T^{1/4-\sigma}+T^{3/8-3\sigma/2}).$$
 About $R_2$, we can change the integration and applying Lemma \ref{lem-2}, then we obtain
 \begin{equation*}\begin{aligned}
 &R_2=R_{21}+R_{22}\\
 &=2\left(\frac{\pi}{T}\right)^{2\sigma-1}\sum_{k\le M}\sum_{l\le M}\Im\left\{e^{-2\pi i\sigma}\frac{a(k)\overline{a(l)}}{[k,l]^{2\sigma}}\right.\left.\int_X^{2X}Q_1(x;\lambda)Q_2(x;\kappa,\lambda)e^{iF(x;\kappa,\lambda)}dx\right\}\\
 &+2\left(\frac{\pi}{T}\right)^{2\sigma-1}\sum_{k\le M}\sum_{l\le M}\Im\left\{e^{-2\pi i\sigma}\frac{a(k)\overline{a(l)}}{[k,l]^{2\sigma}}\int_X^{2X}Q_1(x;\lambda)R(2T,x/\kappa\lambda)dx\right\},
\end{aligned}\end{equation*}
where
\begin{align*}
    &F(x;\kappa\lambda)=2f(T,X/2\kappa\lambda)-\frac{\pi x}{\kappa\lambda}+\frac{3\pi}{4},\\
   & Q_1(x;\lambda)=\frac{\zeta(2-2\sigma)}{\lambda^{2-2\sigma}}+\frac{\zeta(2\sigma)}{\lambda^{2\sigma}}X^{2\sigma-1},\\
   & Q_2(x;\kappa,\lambda)=T^{\sigma-1/2}\left(\frac{x}{\kappa\lambda}\right)^{1/2-\sigma}\left(\arcsinh\sqrt{\frac{\pi x}{4T\kappa\lambda}}\right)^{-1}\left(\frac{2Tx}{\pi\kappa\lambda}+\frac{x^2}{4\kappa^2\lambda^2}\right)^{-1/4}.
\end{align*}
Thus we have 
$$R_{22}\ll \sum_{k\ll M}\sum_{l\ll M}\frac{|a(k)\overline{a(l)}|}{[k,l]^{2\sigma}}T^{\sigma-3/2}\kappa^{\sigma}(\lambda^{3\sigma-2}X^{1-\sigma}+\lambda^{-\sigma}X^{\sigma}).$$

\begin{lemma}(Heath-Brown \cite{HB}).

Let $F(x)$, $G_j(x)$, $1
\le j\le J$, be continuous functions defined on the interval [a,b], which are monotonic. Assume $F'(x)$ is also monotonic, $|F'(x)|\ge C_0^{-1}$ and $|G_j(x)\le C_j|$, $1
\le j\le J$, on [a,b]. Then,
\begin{equation*}\begin{aligned}
\left|\int_a^bG_1(x)\cdots G_J(x)e^{iF(x)}dx\right|\le 2^{J+3}\prod_{j=0}^J C_j.
\end{aligned}\end{equation*}
\label{lem-5}
\end{lemma}
In this case, we can see that 
$$F'(x;\kappa,\lambda)=\sqrt{\frac{4\pi T}{x\kappa\lambda}+\left(\frac{\pi}{\kappa\lambda}\right)^2}-\frac{\pi}{\kappa\lambda}$$
is positive and decreasing, hence we can take $C_0=F'(2X;\kappa,\lambda)^{-1}$. 

Hence by Lemma \ref{lem-5}, we obtain that the integral in the $R_{21}$ is
$$\ll T^{\sigma-1/2}\left(\frac{X}{\kappa\lambda}\right)^{1/2-\sigma}(\lambda^{2\sigma-2}+\lambda^{-2\sigma}X^{2\sigma-1})\kappa\lambda T^{-1/2}.$$
This implies
$$R_{21}\ll_M T^{-\sigma}(X^{1/2-\sigma}+X^{\sigma-1/2}).$$
Collecting all the above results, we obtain $$R_2\ll_MO(T^{-1/2}+T^{1/2-2\sigma}).$$
So we complete the proof of Theorem \ref{thm1}.
\section{The proof of Theorem \ref{thm2}}
Let $L$ be a positive integer satisfying $2^{-L}T\ge C^{\star}$. Then replacing $T$ by $2^{-j}T$ in Theorem~\ref{thm1} for $1\le j\le L$. Summing them up we obtain
\begin{align*}
   \int_0^{T} |\zeta(\sigma+it)A(\sigma+it)|^2dt=\mathcal{M}(T,A)+\Sigma_1(T,Y)+\Sigma_2(T,\xi(T,Y))\\
-\mathcal{M}(2^{-L}T,A)-\Sigma_1(2^{-L}T,2^{-L}Y)-\Sigma_2(2^{-L}T,\xi(2^{-L}T,2^{-L}Y))\\
+O_M\left(\sum_{j=1}^L(2^{-j}T)^{1-2\sigma}\log(2^{-j}T)\right)+\int_0^{2^{-L}T}|\zeta(\sigma+it)A(\sigma+it)|^2dt.
\end{align*}
The condition $2^{-L}T\ge C^{\star}$ is equivalent to 
\begin{equation}L\ll\left[\frac{\log T-\log C^{\star}}{\log2}\right],\label{eq:L}\end{equation}
Since $T\ge C^{\star}\ge e$, we have $\log\log T\ge 0$, therefore the choice 
$$L=\left[\frac{\log T-\log C^{\star}-\alpha\log\log T}{\log2}\right]$$
satisfies \eqref{eq:L} for any positive $\alpha$. 
Thus
$$\mathcal{M}(2^{-L}T,A)\ll_M (2^{-L}T)^{2-2\sigma}\ll_M \log^{(2-2\sigma)\alpha}(T).$$
Also, since $\arcsinh(x)\asymp x$ for small $x$, we have

\begin{align*}
\Sigma_1(2^{-L}T,2^{-L}Y)\ll_M (2^{-L}T)^{1-\sigma}\sum_{k,l\le M}\sum_{n\le \kappa\lambda Y/2^{L}}\sigma_{2\sigma-1}(n)n^{-\sigma-1/2}\\
\ll_M \log^{(1-\sigma)\alpha}(T)(Y2^{-L})^{1/2-\sigma+\varepsilon}\ll_M \log^{(3/2-2\sigma+\varepsilon)\alpha}(T),
\end{align*}
\begin{align*}
\Sigma_2(2^{-L}T,\xi(2^{-L}T,2^{-L}Y))\ll_M (2^{-L}T)^{1/2-\sigma}\sum_{k,l\le M}\sum_{n\le \kappa\lambda Y/2^{L}}\sigma_{2\sigma-1}(n)n^{-\sigma}\\
\ll_M \log^{(1/2-\sigma)\alpha}(T)\log^{(1-\sigma+\varepsilon)\alpha}(T)\ll_M \log^{(3/2-2\sigma+\varepsilon)\alpha}(T),
\end{align*}
where we use the mean square estimate for $\zeta(\sigma+it)$ in \cite{T}: 
$$\int_0^{2^{-L}T}|\zeta(\sigma+it)A(\sigma+it)|^2dt\ll_M(2^{-L}T)^{2-2\sigma}+(2^{-L}T)\ll_M \log^{(2-2\sigma)\alpha}(T)+\log^{\alpha}(T)$$
and 
$$\sum_{j=1}^L(2^{-j}T)^{1-2\sigma}\log(2^{-j}T)\ll LT^{1-2\sigma} \log T\ll T^{1-2\sigma}\log^2 T.$$
So, we obtain
\begin{align*}
\begin{aligned}
      \int_0^{T} |\zeta(\sigma+it)A(\sigma+it)|^2dt-\mathcal{M}(T,A)-\Sigma_1(T,Y)-\Sigma_2(T,\xi(T,Y))\\
      \ll O_M\left(\log^{(2-2\sigma)\alpha}(T)+\log^{(3/2-2\sigma+\varepsilon)\alpha}(T)+\log^{\alpha}(T)\log T+ T^{1-2\sigma}\log^2 T\right).
\end{aligned}\end{align*}
\\

\noindent\textbf{Acknowledgment.} The author would like to thank Professor Kohji Matsumoto and Professor Henrik Bachmann deeply for their helpful comments and valuable advice.

\end{document}